\documentclass[11pt]{amsart}

\usepackage{amsmath}
\usepackage{amssymb,amsfonts}
\usepackage{amsthm}
\usepackage{graphicx}
\usepackage[all,arc]{xy}
\usepackage{hyperref}
\usepackage{enumerate}
\usepackage{bbm}
\usepackage{dsfont}
\usepackage{mathrsfs}
\usepackage{tikz-cd}
\usepackage[left=1in,top=1in,right=1in,bottom=1in]{geometry}
\usepackage{nicefrac}
\usepackage{mathtools}
\usepackage{kbordermatrix}
\renewcommand{\kbldelim}{(}
\renewcommand{\kbrdelim}{)}
\newtheorem{thm}{Theorem}[section]
\newtheorem*{thm*}{Theorem}

\newtheorem{innercustomthm}{Theorem}
\newenvironment{customthm}[1]
{\renewcommand\theinnercustomthm{#1}\innercustomthm}
{\endinnercustomthm}
\newtheorem{cor}[thm]{Corollary}
\newtheorem{prop}[thm]{Proposition}
\newtheorem{lem}[thm]{Lemma}

\theoremstyle{definition}
\newtheorem{defn}[thm]{Definition}
\newtheorem{rem}[thm]{Remark}

\newtheorem{exmp}[thm]{Example}

\newtheorem{notn}[thm]{Notation}

\newcommand{\Z}{\mathbb{Z}}
\newcommand{\Q}{\mathcal{Q}}

\newcommand{\cat}{\mathsf}

\newcommand{\sS}{\mathfrak{S}}

\newcommand{\R}{\mathbf{R}}
\newcommand{\D}{\mathcal{D}}
\newcommand{\Aut}{\operatorname{Aut}}

\newcommand{\op}{\mathrm{op}}

\newcommand{\trop}{\mathrm{trop}}

\renewcommand{\arraystretch}{2}
\DeclareSymbolFont{bbold}{U}{bbold}{m}{n}
\DeclareSymbolFontAlphabet{\mathbbold}{bbold}
\newcommand{\defeq}{\vcentcolon=}

\newcommand{\G}{\mathbf{G}}
\newcommand{\Hom}{\operatorname{Hom}}

\newcommand{\overbar}[1]{\mkern 1.5mu\overline{\mkern-1.5mu#1\mkern-1.5mu}\mkern 1.5mu}
\newcommand{\val}{\operatorname{val}}

\newcommand{\B}{\mathbf{B}}

\newcommand{\tor}{\mathrm{tor}}

\newcommand{\EL}{\mathrm{EL}}

\renewcommand{\O}{\mathcal{O}}
\newcommand{\V}{\mathsf{V}}

\renewcommand{\H}{\mathbf{H}}

\newcommand{\Stab}{\operatorname{Stab}}
\renewcommand{\sS}{\mathfrak{S}}

\makeatletter
\let\c@equation\c@thm
\makeatother
\numberwithin{equation}{section}

\graphicspath{ {autdelta-submission-figs/} }
\title{Symmetries of tropical moduli spaces of curves}

\author{Siddarth Kannan}
\address{Department of Mathematics, Brown University, Providence, RI 02906}
\email{\url{siddarth_kannan@brown.edu}}

\begin{document}

	\begin{abstract}
	 Put $\Delta_{g, n} \subset M_{g, n}^\trop$ for the moduli space of stable $n$-marked tropical curves of genus $g$ and volume one. We compute the automorphism group $\Aut(\Delta_{g, n})$ for all $g, n \geq 0$ such that $3g - 3 + n > 0$. In particular, we show that $\Aut(\Delta_{g})$ is trivial for $g \geq 2$, while $\Aut(\Delta_{g, n}) \cong S_n$ when $n \geq 1$ and $(g, n) \neq (0, 4), (1, 2)$. The space $\Delta_{g, n}$ is a \textsl{symmetric $\Delta$-complex} in the sense of Chan, Galatius, and Payne ~\cite{CGP1, CGP2}, and is identified with the dual intersection complex of the boundary divisor in the Deligne-Mumford-Knudsen moduli stack $\overbar{\mathcal{M}}_{g, n}$ of stable curves. After the work of Massarenti ~\cite{Mass}, who has shown that $\Aut(\overbar{\mathcal{M}}_g)$ is trivial for $g \geq 2$ while $\Aut(\overbar{\mathcal{M}}_{g, n}) \cong S_n$ when $n \geq 1$ and $2g - 2 + n \geq 3$, our result implies that the tropical moduli space $\Delta_{g, n}$ faithfully reflects the symmetries of the algebraic moduli space for general $g$ and $n$.
	\end{abstract}
	
	\maketitle
	\tableofcontents
\section{Introduction}
The moduli spaces $M_{g, n}^\trop$ of $n$-marked tropical curves of genus $g$ and their extended counterparts $\overbar{M}_{g, n}^\trop$ have been extensively studied in recent years, with particular interest in the relationship between these spaces and the Deligne-Mumford-Knudsen compactifications $\mathcal{M}_{g, n} \subset \overbar{\mathcal{M}}_{g, n}$ of the moduli stacks of algebraic curves. Abramovich, Caporaso, and Payne \cite{ACP} identified the space $\overbar{M}_{g, n}^\trop$ with the skeleton of the toroidal embedding $\mathcal{M}_{g, n} \subset \overbar{\mathcal{M}}_{g, n}$. Recently, Chan, Galatius, and Payne \cite{CGP1, CGP2} realized the subspace $\Delta_{g, n} \subset M_{g, n}^\trop$ parameterizing tropical curves of volume one as the dual complex of the boundary divisor $\overbar{\mathcal{M}}_{g, n} \smallsetminus \mathcal{M}_{g, n}$, and used this identification to study the top weight rational cohomology of $\mathcal{M}_{g, n}$ via the topology of $\Delta_{g, n}$. In this paper, we compute the automorphism groups of each of these related tropical moduli spaces for all $g, n \geq 0$ such that $3g - 3 + n > 0$. For $n \geq 1$, we set $I_n \defeq \{1, \ldots, n \}$
and put $S_n \defeq \mathrm{Perm}(I_n)$ for the group of permutations of $I_n$. We also adopt the convention that $S_0$ is the trivial group.
\begin{thm}\label{Main}
Fix integers $g, n \geq 0$ such that $3g - 3 + n > 0$. If $2g - 2 + n \geq 3$, then 
    \[\Aut(\overbar{M}_{g, n}^\trop) \cong \Aut(M_{g, n}^\trop) \cong \Aut(\Delta_{g, n}) \cong S_n. \]
If $2g - 2 + n < 3$, then $(g, n) \in \{(0, 4), (1, 1), (1, 2)\}$. In these cases we have $\Aut(\Delta_{0, 4}) \cong S_3$ while $\Aut(\Delta_{1,1})$ and $\Aut(\Delta_{1, 2})$ are both trivial.
\end{thm}
The automorphism groups of $\Delta_{g, n}$, $M_{g, n}^\trop$, and $\overbar{{M}}_{g, n}^\trop$ are taken in the categories of \textsl{symmetric $\Delta$-complexes}, \textsl{generalized cone complexes}, and \textsl{extended generalized cone complexes}, respectively. The case $g = 0$ in Theorem \ref{Main} is due to Abreu and Pacini \cite{AbreuPacini}, and our work specializes to give a new proof of their result. With respect to the algebraic moduli spaces, Massarenti \cite{Mass} has shown that the automorphism group of $\overbar{\mathcal{M}}_{g, n}$ is isomorphic to $S_n$ for all $g \geq 0$ and $n \geq 1$ such that $2g - 2 + n \geq 3$, while $\Aut(\overbar{\mathcal{M}}_g)$ is trivial for $g \geq 2$. The automorphism group of $\overline{\mathcal{M}}_{0, n}$ was also computed independently by Bruno and Mella in ~\cite{BM}. Thus Theorem \ref{Main} is yet another example of how the combinatorics of the skeleta $\overbar{M}_{g, n}^\trop$ reflect the geometry of the moduli stacks $\overbar{\mathcal{M}}_{g, n}$ in a meaningful way.
\subsection{Tropical moduli spaces} The space $M_{g, n}^\trop$ parameterizes $n$-marked tropical curves of genus $g$. Such a curve is a pair $(\G, \ell)$, where $\G$ is a stable $n$-marked graph of genus $g$ (see Section \ref{Background} for a formal definition) and $\ell: E(\G) \to \mathbb{R}_{> 0}$ is a function assigning positive real lengths to the edges of $\G$. Thus the building blocks of $M_{g, n}^\trop$ are quotient spaces of cones by automorphism groups of graphs, i.e. cells of the form
\[C(\G) = \mathbb{R}_{\geq 0}^{E(\G)}/\Aut(\G). \]
It is in this way that $M_{g, n}^\trop$ is given the structure of a \textsl{generalized cone complex}. The tropical moduli space admits a compactification
\[M_{g, n}^\trop  \subset \overbar{M}_{g, n}^\trop \] by an \textsl{extended} generalized cone complex, which is obtained by allowing edge lengths of tropical curves to become infinite. See \cite{BMV, Caporaso} for the construction of $M_{g, n}^\trop$, \cite{ACP} for basics on (extended) generalized cone complexes, and \cite{cavalieri2017moduli} for foundations on tropical moduli problems. The third combinatorial moduli space in which we are interested is the subspace $\Delta_{g, n} \subset M_{g, n}^\trop$ parameterizing tropical curves of volume one, where the \textsl{volume} of a tropical curve $(\G, \ell)$ is defined to be
\[\mathrm{vol}(\G, \ell) \defeq \sum_{e \in E(\G)}\ell(e). \]
Besides its role as a moduli space, $\Delta_{g, n}$ is also identified with the dual complex of the normal crossings boundary divisor $\overbar{\mathcal{M}}_{g, n} \smallsetminus \mathcal{M}_{g, n}$. The dual complex of a simple normal crossings divisor $D$, first studied by Danilov \cite{Danilov}, is a $\Delta$-complex which encodes the combinatorial data of the components of $D$ and their intersections. The construction of dual complexes can been extended to the generality of normal crossings divisors on Deligne-Mumford stacks ~\cite{ACP, CGP1, Harper}. The resulting combinatorial objects are generalizations of $\Delta$-complexes called \textsl{symmetric} $\Delta$-complexes. The building blocks of a symmetric $\Delta$-complex are quotient spaces of simplices by subgroups of the symmetric group. In \cite{CGP1}, it is proven that the category of symmetric $\Delta$-complexes is equivalent to the category of \textsl{smooth} generalized cone complexes, as defined in \cite{ACP}. This fact leads us to the following lemma, which allows us to focus our attention on calculating $\Aut(\Delta_{g, n})$.
\begin{lem}\label{EquivalenceOfAutomorphisms}
	For all $g, n \geq 0$ with $3g - 3 + n > 0$, there are canonical isomorphisms
	\[ \Aut(\overbar{M}_{g, n}^\trop)\cong \Aut(M_{g, n}^\trop)\cong\Aut(\Delta_{g, n}). \]
\end{lem}
\begin{proof}
	It is shown in \cite{CGP1, CGP2} that $\Delta_{g, n}$ is identified with the link of the cone point in $M_{g, n}^\trop$, and that taking the link gives rise to an equivalence of categories between smooth generalized cone complexes and symmetric $\Delta$-complexes. As a result, we get a canonical isomorphism $\Aut(\Delta_{g, n}) \cong \Aut(M_{g, n}^\trop)$. Any automorphism of $M_{g, n}^\trop$ extends by linearity to one of $\overbar{M}_{g, n}^\trop$, giving a map $\Aut(M_{g, n}^\trop) \to \Aut(\overbar{M}_{g, n}^\trop)$. That this map is an isomorphism is worked out when $g = 0$ by Abreu and Pacini \cite{AbreuPacini}, and their proof carries over verbatim to the general case.
\end{proof}	
\subsection{Skeletons of toroidal embeddings} Following Thuillier's work \cite{Thuillier} on toroidal schemes, Abramovich, Caporaso, and Payne show in \cite{ACP} how to associate an extended generalized cone complex $\overbar{\Sigma}(\mathcal{X})$ to a toroidal embedding $\mathcal{U} \subset \mathcal{X}$ of Deligne-Mumford stacks. Moreover, they show that this construction is functorial with respect to toroidal morphisms of toroidal stacks, and study $\overbar{\Sigma}(\mathcal{X})$ when $\mathcal{X} = \overbar{\mathcal{M}}_{g, n}$ and $\mathcal{U} = \mathcal{M}_{g, n}$. With respect to this toroidal structure, they show that \[\overbar{\Sigma}(\overbar{\mathcal{M}}_{g, n}) = \overbar{M}_{g, n}^\trop.\] If we put $\Aut^{\tor}(\mathcal{X})$ for the group of toroidal automorphisms of the toroidal stack $\mathcal{X}$, the functoriality of the construction of skeletons gives a group homomorphism $\Aut^{\tor}(\mathcal{X}) \to \Aut(\overbar{\Sigma}(\mathcal{X}))$. When $\mathcal{X} = \overbar{\mathcal{M}}_{g, n}$, Massarenti's work gives \textsl{a posteriori} that
\[\Aut(\overbar{\mathcal{M}}_{g, n}) \cong \Aut^\tor(\overbar{\mathcal{M}}_{g, n}) \cong S_n \]
for $2g - 2 + n \geq 3$, since $S_n$ preserves $\mathcal{M}_{g, n}$ and therefore acts toroidally. This yields a map of groups
\begin{equation}\label{ToroidalAction}
S_n \to \Aut(\overbar{M}_{g, n}^\trop),
\end{equation}
and the content of Theorem \ref{Main} is that this map is an isomorphism.
\subsection{Connection with the complex of curves and Outer space} Suppose $g \geq 2$, $n \geq 0$, and let $\Sigma_{g, n}$ be a smooth surface of genus $g$ with $n$ punctures. Then the \textsl{complex of curves} $\mathcal{C}(\Sigma_{g, n})$, first considered by Harvey ~\cite{harvey1981boundary}, is an abstract simplicial complex with vertices given by the set of free isotopy classes $[\gamma]$ of essential simple closed curves on $\Sigma_{g, n}$. There is a $k$-simplex on set of $k + 1$ vertices in $\mathcal{C}(\Sigma_{g, n})$ whenever the corresponding curve classes may be realized disjointly in $\Sigma_{g, n}$. The complex $\mathcal{C}(\Sigma_{g, n})$ admits an action of the mapping class group $\mathrm{Mod}_{g, n}$ of orientation-preserving self-diffeomorphisms of $\Sigma_{g, n}$ modulo isotopy. The \textsl{pure} mapping class group $\mathrm{PMod}_{g, n}$ is defined as the subgroup of $\mathrm{Mod}_{g, n}$ which fixes each puncture individually. Upon choosing an ordering of the punctures, one has a homeomorphism \cite{CGP2}
\begin{equation*}
\mathcal{C}(\Sigma_{g, n}) / \mathrm{PMod}_{g, n} \cong \Delta_{g, n},
\end{equation*}
connecting the study of tropical moduli spaces with curve complexes. In this way our Theorem \ref{Main} can be compared with the work of Ivanov \cite{Ivanov} and Luo \cite{Luo}, who find that the automorphism group of $\mathcal{C}(\Sigma_{g, n})$ in the simplicial category is given by the \textsl{extended} mapping class group $\mathrm{Mod}^{\pm}_{g, n}$ of \textsl{all} self-diffeomorphisms modulo isotopy: \[\Aut(\mathcal{C}(\Sigma_{g, n})) \cong \mathrm{Mod}^{\pm}_{g, n}.\] An analogous relationship exists between the moduli space of unmarked tropical curves $\Delta_{g}$ and the simplicial completion $\overbar{X}_g$ of the Culler-Vogtmann Outer space $X_g$ \cite{CV86}. Namely, $\Delta_{g}$ is the topological quotient of $\overbar{X}_g$ by the action of the outer automorphism group $\mathrm{Out}(F_g)$ of the free group on $g$ generators. Indeed, our proof of Theorem \ref{Main} is similar in some aspects to the work of Bridson and Vogtmann \cite{bridson2001}, who have shown that
\[\Aut(\mathsf{K}_g) \cong \mathrm{Out}(F_g) \]
for $g \geq 3$, where $\mathsf{K}_g$, called the \textsl{spine} of Outer space, is an $\mathrm{Out}(F_g)$-equivariant, simplicial deformation retract of $X_g$. See ~\cite{CMV} for more on the connection between Outer space and tropical moduli spaces.

\subsection{Outline of the paper.} 
We recall the definition of a symmetric $\Delta$-complex and the construction of $\Delta_{g, n}$ in Section \ref{Background}. In Section \ref{RestrictionExistence}, we study the filtration
\[ \V^{1}_{g, n} \subset \V^{2}_{g, n} \subset \cdots \subset \V^{2g -2 + n}_{g, n} = \Delta_{g, n}, \]
where $\V^{i}_{g, n}$ is the subcomplex of $\Delta_{g, n}$ parameterizing graphs with at most $i$ vertices. We will show that this filtration is intrinsic to the simplicial structure of $\Delta_{g, n}$, in that it is preserved by automorphisms. In this way we get a restriction homomorphism
\begin{equation}\label{Restriction}
\rho_{g, n}^{i}: \Aut(\Delta_{g, n}) \to \Aut(\V_{g, n}^{i})
\end{equation}
for each $i$ such that $1 \leq i \leq 2g - 2 + n$. In Section \ref{RestrictionInjects}, we prove the main technical result of the paper, which is that the restriction map $\rho^i_{g, n}$ is injective for $i = 2$.
\begin{thm}\label{TwoVertex}
	Let $g, n \geq 0$ such that $2g - 2 + n \geq 3$. Then the restriction map
	\[\rho_{g, n}^{2}: \Aut(\Delta_{g, n}) \to \Aut(\V^2_{g, n}) \]
	is an injection.
\end{thm}
 Given $\sigma \in S_n$, we put $f_\sigma \in \Aut(\Delta_{g,n})$ for the result of composing (\ref{ToroidalAction}) with the isomorphism $\Aut(\overbar{M}_{g, n}^\trop) \cong \Aut(\Delta_{g,n})$ of Lemma \ref{EquivalenceOfAutomorphisms}.  In Section \ref{Calculations}, we will prove the following result, which together with Theorem \ref{TwoVertex} implies Theorem \ref{Main} when $2g - 2 + n \geq 3$.
\begin{thm}\label{Image}
	Suppose that $g, n \geq 0$ and $2g - 2 + n \geq 3$, and let $\Phi \in \Aut(\Delta_{g, n})$. If $n = 0$, so $g \geq 2$, then we have
	\[\Phi|_{\V^{2}_g} = \mathrm{Id}|_{\V^2_g}. \]
	If $n \geq 1$, then there exists a unique element $\sigma \in S_n$ such that 
	\[\Phi|_{\V^{2}_{g, n}} = f_\sigma|_{\V^2_{g, n}}. \] 
\end{thm}
The remaining values of $(g, n)$ appearing in the statement of Theorem \ref{Main} land in the set
\[\{(0, 4), (1, 1), (1, 2), (2, 0) \}; \] 
these cases are dealt with in Example \ref{SmallExample}.

\subsection{Acknowledgements} I am grateful to Melody Chan for suggesting this problem and advising me along the way. I also thank Dan Abramovich for looking over an early version of this draft, and Sam Freedman for several valuable discussions. This work was partially supported by NSF DMS-1701659 and an NSF Graduate Research Fellowship.
\section{Graph categories and the definition of $\Delta_{g,n}$}\label{Background}
In this section we describe the structure of $\Delta_{g, n}$ as a symmetric $\Delta$-complex, via an interlude on graph categories. Our definitions of the categories $\Gamma_g$ and $\Gamma_{g, n}$ are equivalent to those in ~\cite{CGP1, CGP2}; the few differences in presentation are purely superficial. Discussions of stable graphs of genus $g$ and the morphisms between them have become standard in the study of moduli spaces of curves ~\cite{Arbarello, Getzler, Kontsevich}. Similar categories have recently been studied in a representation-theoretic context by Proudfoot and Ramos ~\cite{PR1, PR2}.
\subsection{Graph categories.} By a \textbf{\textsl{graph}} $G$ we mean a finite set $X(G) = V(G) \sqcup H(G)$ of \textsl{\textbf{vertices}} and \textsl{\textbf{half-edges}}, together with a pair of maps $s_G, r_G: X(G) \to X(G)$ with $s_G^2 = \mathrm{Id}$, $r_G^2 = r_G$, and such that  
\[ \{x \in X(G) \mid r_G(x) = x \} = \{x \in X(G) \mid s_G(x) = x \} = V(G).\]
The function $r_G$ sends half-edges to their incident vertex, while $s_G$ couples half-edges. A graph $G$ has a \textsl{\textbf{geometric realization}} $|G|$ which is a finite one-dimensional CW complex; we only work with \textsl{\textbf{connected}} graphs, which are those whose geometric realization is connected as a topological space. The \textbf{\textsl{genus}} $b^1(G)$ of a graph is the first Betti number of the space $|G|$. An \textsl{\textbf{edge}} of $G$ is an element of the quotient set
\[E(G) \defeq H(G) / \{s_G(x) \sim x \}. \]
Then we have
\[ b^1(G) = |E(G)| - |V(G)| + 1, \]
and
\[2|E(G)| = |H(G)|. \]
Note that our definition allows for \textsl{\textbf{loops}}, which are those elements $e \in E(G)$ such that $s_G$ is constant on the half-edges $h_1, h_2 \in H(G)$, where $h_1, h_2$ are the two half edges of $e$.
\begin{defn}
A \textsl{\textbf{weighted graph of genus $g$}} is a pair $(G, w)$ where $G$ is a connected graph, and $w: V(G) \to \Z_{\geq 0}$ is a function, such that
\[b^1(G) + \sum_{v \in V(G)} w(v) = g. \]
\end{defn} 
We will make a category $\mathsf{Graph}_{g}$ whose objects are weighted graphs of genus $g$. The morphisms will consist of isomorphisms and edge-contractions. 
\begin{defn}
	An \textbf{\textsl{isomorphism}} $\varphi: (G, w) \to (G', w')$ of weighted graphs of genus $g$ is given by a bijection $\varphi: X(G) \to X(G)$ such that 
	\begin{itemize}
		\item$r_G' \circ \varphi = \varphi \circ r_G$,
		\item $s_G' \circ \varphi = \varphi \circ s_G$, and
		\item $w'(\varphi(v)) = w(v)$ for all $v \in V(G)$.
	\end{itemize}
\end{defn}
 
\begin{defn} 
Given $(G, w)$ a weighted graph of genus genus $g$ and an edge $e \in E(G)$, the \textsl{\textbf{edge-contraction}} $(G/e, w_e)$ is a weighted graph $(G/e, w_e)$, together with a surjective map $c_e: X(G) \to X(G/e)$, which is determined as follows.
\begin{itemize}
\item There is an identification
\[H(G/e) = H(G) \smallsetminus \{h_1, h_2 \}, \]
where $h_1, h_2$ are the two half-edges uniquely determining $e$.  Using this identification, set 
\[{s_{G/e}}_{|_{H(G/e)}} = {s_{G}}_{|_{H(G/e)}}. \]
\item If $e$ is a loop, we set $V(G/e) = V(G)$,
\[{r_{G/e}}_{|_{H(G/e)}} = {r_{G}}_{|_{H(G/e)}}. \]
 We also put
\[ w_e(v) = w(v) + 1 \]
if $v = r_G(h_1) = r_G(h_2)$,
and $w_e(v) = w(v)$ otherwise.
\item If $e$ is a non-loop edge, then we set
\[V(G/e) =(V(G) \smallsetminus \{r_G(h_1), r_G(h_2) \}) \cup \{\hat{v}\}. \] Using this identification, we set
\[{r_{G/e}}(h) = \begin{cases}
r_{G}(h) &\text{if } r_G(h) \notin \{r_G(h_1), r_G(h_2) \}\\
\hat{v} &\text{otherwise},
\end{cases} \]
when $h \in H(G/e) = H(G) \smallsetminus \{h_1, h_2 \}$. We also put $w_e(v) = w(v)$ if $v \in V(G)$, and we set \[w_e(\hat{v}) = w(r_G(h_1)) + w(r_G(h_2)).\]
The graph $G/e$ is equipped with a surjective map
\[c_e: X(G) \to X(G/e), \]
called the \textsl{\textbf{contraction map}}, which satisfies \[c_e(h_1) = c_e(h_2) =  c_e(r_G(h_1)) = c_e(r_G(h_2)) = \hat{v},\]
and $c_e(x) = x$ otherwise.
\end{itemize}
Informally, the pair $(G/e, w_e)$ is obtained from $(G, w)$ by contracting the edge $e$ to a point, and adjusting weight of the new vertex so that $(G/e, w_e)$ is still a weighted graph of genus $g$, i.e. adding weight one if $e$ is a loop, and adding the two weights coming from the endpoints of $e$ otherwise. 
\end{defn}
We are now ready to define the category $\mathsf{Graph}_{g}$.
\begin{defn}
	The category $\mathsf{Graph}_{g}$ has as its objects weighted graphs $(G, w)$ of genus $g$, and morphisms $(G, w) \to (G', w')$ given by maps of sets $X(G) \to X(G')$ which factor as compositions of isomorphisms and contraction maps. 
\end{defn}
We note that the processes of passing to sets of half-edges or edges are both contravariant with respect to morphisms in $\mathsf{Graph}_g$.
\begin{defn}
	We define functors \[H, E : \mathsf{Graph}_g^\mathrm{op} \to \mathsf{Set}\] on objects by taking sets of half-edges and edges, respectively. Given a morphism $f: G \to G'$ and $h' \in H(G')$, there is a unique element $h \in H(G)$ such that $f(h) = h'$, and we set $H(f)(h') = h$. Similarly, an edge $e' \in E(G')$ has a unique preimage $e \in E(G)$, and we set $E(f)(e') = e$.
\end{defn}
Given a morphism $f: (G, w) \to (G', w')$ of weighted graphs of genus $g$, we use the notation $f^* = E(f): E(G') \to E(G)$. The category $\mathsf{Graph}_g$ has infinitely many isomorphism classes of objects. We pass to a category $\Gamma_g$ which has only finitely many isomorphism classes via the algebro-geometric stability condition.
\begin{defn}
We say a weighted graph $(G, w)$ of genus $g$ is \textsl{\textbf{stable}} if for all $v \in V(G)$, we have
\[2w(v) - 2  + \val(v) > 0, \]
where
\[\val(v) = |\{h \in H(G) \mid r_G(h) = v \}|.  \]
Equivalently, $\val(v) \geq 3$ whenever $w(v) = 0$ and $\val(v) \geq 1$ whenever $w(v) = 1$. We put $\Gamma_g$ for the full subcategory of $\mathsf{Graph}_g$ determined by the stable graphs; the objects of $\Gamma_g$ are called \textsl{\textbf{stable weighted graphs of genus $g$}}.
\end{defn}
Note that $\Gamma_g$ contains objects only when $g \geq 2$. We use the boldface notation $\G = (G, w)$ for the objects of $\Gamma_g$, reserving the notation $(G, w)$ for those instances where it is not required that the object we are working with is stable. We now define a marked analogue $\Gamma_{g, n}$ of the category $\Gamma_g$. Recall that we have set $I_n = \{1, \ldots, n \}$.
\begin{defn}
	Suppose that $n \geq 1$ and $2g - 2 + n > 0$. A \textsl{\textbf{stable $n$-marked weighted graph of genus $g$}} is a triple $\G = (G, w, m)$ where $(G, w)$ is a weighted graph of genus $g$, and $m: I_n \to V(G)$ is a function, such that for all $v \in V(G)$, we have
	\[2w(v) - 2 + \val(v) + |m^{-1}(v)| > 0; \]
	equivalently, we have $\val(v) + |m^{-1}(v)| \geq 3$ when $w(v) = 0$, and $\val(v) + |m^{-1}(v)| \geq 1$ when $w(v) = 1$. An \textsl{\textbf{isomorphism}} $\varphi: \G \to \G'$ is an isomorphism $\varphi: (G, w) \to (G', w')$ of weighted graphs which satisfies further that $m'(s) = (\varphi \circ m)(s)$ for all $s \in I_n$. Given an edge $e \in E(G)$, the \textsl{\textbf{edge-contraction}} of $e$ is an object $\G/e = (G/e, w_e, m_e)$ in $\Gamma_{g, n}$, such that $c_e: (G, w) \to (G/e, w_e)$ is the edge-contraction of $e$ in the category of weighted graphs of genus $g$, and $m_e$ is determined as follows.
	\begin{itemize}
		\item If $e$ is a loop, so we have an identification $V(G/e) = V(G)$, we put $m_e^{-1}(v) = m^{-1}(v)$ for all $v \in V(G/e)$.
		\item If $e$ is a non-loop edge corresponding to the two half-edges $h_1, h_2$, so we have \[V(G/e) = (V(G) \smallsetminus \{r_G(h_1), r_G(h_2) \}) \cup \{\hat{v} \},\] then $m_e$ is uniquely determined by the property that $m_e^{-1}(v) = m^{-1}(v)$ whenever $v \in V(G)$, and
		\[m_e^{-1}(\hat{v}) = m^{-1}(r_G(h_1)) \cup m^{-1} (r_G(h_2)). \] 
	\end{itemize}	
	We put $\Gamma_{g, n}$ for the category whose objects are stable $n$-marked graphs of genus $g$, and whose morphisms $\G = (G, w, m) \to \G' = (G', w', m')$ are given by those maps of sets $X(G) \to X(G')$ which factor as compositions of edge-contractions and isomorphisms.
\end{defn}
For convenience we adopt the convention that $\Gamma_{g, 0} = \Gamma_g$. For each $n \geq 0$, there is a forgetful functor \[\Gamma_{g, n} \to \mathsf{Graph}_g,\] simply given by forgetting the marking function. Given an object $\G = (G, w, m)$ of $\Gamma_{g, n}$, we put $H(\G), E(\G), V(\G)$ for the sets $H(G)$, $E(G)$, $V(G)$, respectively. We can view this as composing the functors $H, E: \mathsf{Graph}_g^\mathrm{op}\to\mathsf{Set}$ with the opposite of the forgetful functor $\Gamma_{g, n} \to \mathsf{Graph}_g$. We also use the notation $r_\G = r_G$ and $s_\G = s_G$; when $n \geq 1$ we put $m_\G : I_n \to V(\G)$ for the marking function of $\G$.

Ultimately, the simplices of $\Delta_{g, n}$ will correspond to equivalence classes of pairs $(\G, \tau)$ where $\G$ is a $\Gamma_{g, n}$-object and $\tau$ is a labelling of the edges of $\G$. It shall be useful to establish some language to deal with such pairs, and we do so by making an auxiliary groupoid $\Gamma_{g, n}^{\mathrm{EL}}$. For each $p \geq 0$, we define $[p] \defeq \{0, \ldots, p\}$, and formally set $[-1] \defeq \varnothing$.

\begin{defn}\label{ELDefn}
	Suppose $g, n \geq 0$ with $3g - 3 + n > 0$. A \textbf{\textsl{stable, $n$-marked, edge-labelled pair of genus $g$}}, or simply \textsl{\textbf{edge-labelled pair}}, is a tuple $(\G, \tau)$ where $\G$ is a stable $n$-marked graph of genus $g$ and $\tau: E(\G) \to [p]$ is an edge-labelling for some $p \geq -1$. An \textbf{\textsl{isomorphism}} 
	\[\varphi: (\G, \tau) \to (\G', \tau')\] of edge-labelled pairs is an isomorphism $\varphi:\G \to \G'$ in $\Gamma_{g, n}$ which makes the triangle
	\begin{equation}\label{PairIsomorphism}
	\begin{tikzcd}
	&E(\G') \arrow[dr, "\tau'"] \arrow[rr, "\varphi^*"] & &E(\G) \arrow[dl, "\tau"']\\
	& &\lbrack p \rbrack &
	\end{tikzcd}
	\end{equation}
	commute. We say $(\G, \tau)$ and $(\G', \tau')$ are \textsl{\textbf{isomorphic}} if there exists an isomorphism of pairs between them. We use the notation $\Gamma_{g, n}^{\mathrm{EL}}$ for the groupoid whose objects are all stable $n$-marked edge-labelled pairs of genus $g$, and isomorphisms given as above.
\end{defn}
\begin{notn}\label{Skeleta}
	The categories $\Gamma_{g, n}$ and $\Gamma_{g, n}^{\mathrm{EL}}$ admit finite skeletons. To avoid any set-theoretic issues when we define $\Delta_{g,n}$, we now choose compatible skeletons $\mathsf{\Gamma}_{g, n}$ of $\Gamma_{g, n}$ and $\mathsf{\Gamma}^{\EL}_{g, n}$ of $\Gamma_{g,n}^{\EL}$ for all $g, n$. We shall use the notation $[\G, \tau]$ for the unique object of $\mathsf{\Gamma}_{g, n}^{\EL}$ to which $(\G, \tau)$ is isomorphic. So in general we have $[\G, \tau] = [\G', \tau']$ if and only if $(\G, \tau) \cong (\G', \tau')$ in $\Gamma_{g, n}^{\mathrm{EL}}$. For each $p \geq -1$, we also set
	\[\mathsf{\Gamma}_{g, n}(p) = \{\G \in \mathrm{Ob}(\mathsf{\Gamma}_{g, n}) \mid |E(\G)| = p + 1 \}, \]
	and define $\mathsf{\Gamma}_{g, n}^{\EL}(p)$ similarly. Note that there is a unique element of $\mathsf{\Gamma}_{g, n}(-1)$, corresponding to the final object of $\Gamma_{g,n}$.
\end{notn}
\subsection{$\Delta_{g, n}$ as a functor}
Let FI be the category whose objects are finite sets, including the empty set, and whose morphisms are injections. Let I be the full subcategory on the objects $[p]$ for $p \geq -1$.
\begin{defn}
A \textsl{\textbf{symmetric $\Delta$-complex}} is a functor $X: \mathrm{I}^\mathrm{op} \to \mathsf{Set}$. A morphism of symmetric $\Delta$-complexes is a natural transformation of functors.
\end{defn}
Symmetric $\Delta$-complexes have also been called \textsl{symmetric semi-simplicial sets} in the literature. Recall that a \textsl{\textbf{semi-simplicial set}} is a functor $\mathrm{I}_\mathrm{ord}^\mathrm{op} \to \mathsf{Set}$, where $\mathrm{I}_\mathrm{ord}$ is the category whose objects are the same as $\mathrm{I}$, but whose morphisms are only those injections which are order-preserving. A semi-simplical set is thought of as a blueprint for assembling a topological simplicial complex by gluing together $p$-simplices. Symmetric $\Delta$-complexes generalize these by allowing for cells of the form
\[ \sigma^p/H, \] where $\sigma^p$ is a $p$-simplex, and $H$ is a subgroup of permutations of its vertices. There is a \textbf{\textsl{geometric realization functor}}
\[|\cdot|: \mathsf{Symmetric}\,\Delta\text{-}\mathsf{complexes} \to \mathsf{Top} \] 
which we now describe. For each $p \geq 0$, choose a $p$-simplex $\sigma^p$ together with a bijection between its vertices and the set $[p]$. Given this information, any injection $\iota: [p] \to [q]$ uniquely induces a simplicial map $\iota_*: \sigma^p \to \sigma^q$, which includes $\sigma^p$ as a face of $\sigma^q$. With these conventions fixed, the geometric realization of a symmetric $\Delta$-complex $X$ is given by
\begin{equation}\label{Realization}
|X| = \frac{\left(\coprod_{ p \geq 0} X[p] \times \sigma^p\right)}{\sim};
\end{equation}
where the equivalence relation is generated by
\[ (X(\iota)(x), a) \sim (x, \iota_*(a)), \]
whenever $\iota\in \Hom_{\mathrm{I}}([p], [q])$, $x \in X[q]$, and $a \in \sigma^p$. The additional data of $X[-1]$ also equips $|X|$ with a continuous augmentation map $|X| \to X[-1]$. This augmentation is necessary for the equivalence of categories between smooth generalized cone complexes and symmetric $\Delta$-complexes discussed in the introduction; see ~\cite[\S{4.3}]{CGP1}.

We proceed with the description of the functor
\[\Delta_{g, n}: \mathrm{I}^\mathrm{op} \to \mathsf{Set}. \] 
For all $g, n \geq 0$ such that $3g - 3 + n > 0$, we set \[\Delta_{g,n}[p] = \mathsf{\Gamma}_{g, n}^{\EL}(p), \]
as in Notation \ref{Skeleta}. Given an injection $\iota: [p] \to [q]$ and a simplex $[\G, \tau] \in \Delta_{g, n}[p]$, we define $\Delta_{g,n}(\iota)([\G, \tau])$ by
\[\Delta_{g,n}(\iota)([\G, \tau]) = [\G/\tau^{-1}(S), \tau_{\iota}],\] where $S = \tau^{-1}([q]\smallsetminus \iota([p]))$, $\G/\tau^{-1}(S)$ is the graph in $\Gamma_{g, n}$ obtained from $\G$ by sequentially contracting the elements of $\tau^{-1}(S)$ (the order in which this is done is irrelevant), and $\tau_\iota$ is the unique edge labelling $E(\G/\tau^{-1}(S)) \to [p]$ making the diagram
\begin{equation}\label{FaceMaps}
\begin{tikzcd}
&E(\G) \arrow[r, "\tau"] & \lbrack q \rbrack\\
&E(\G/\tau^{-1}(S)) \arrow[u, "(c_S)^*"] \arrow[r, "\tau_\iota"] & \lbrack p \rbrack \arrow[u, "\iota"]
\end{tikzcd} 
\end{equation}
commute,
where 
\[c_S: \G \to \G/\tau^{-1}S \]
is the composition of edge-contractions. The choice of $(\G, \tau)$ representing $[\G, \tau]$ is immaterial: if we have $(\G, \tau) \cong (\G', \tau')$ in $\Gamma_{g, n}^{\EL}$, then $(\G/\tau^{-1}(S), \tau_{\iota}) \cong (\G'/\tau'^{-1}(S), \tau_{\iota}')$ as well.

The topology of the geometric realization $|\Delta_{g,n}|$ has proven to be a rich area of study. See Figure \ref{GeomExample} for some simple examples, and ~\cite{allcock, CGP1, CGP2, chan2015topology, Vogtmann} for general results. We now make some basic observations and fix notation that will be used throughout the paper. Taking $[p] = [q]$ and setting \[\sS_{p + 1} \defeq \Hom_{\mathrm{I}}([p], [p]), \] we see that there is a $\sS_{p + 1}$-action on $\Delta_{g, n}[p]$ for each $p$, given by \[\mathfrak{a} \cdot [\G, \tau] = \mathfrak{a}^*([\G, \tau])\defeq \Delta_{g,n}(\mathfrak{a}^{-1})([\G, \tau]) = [\G, \mathfrak{a} \circ \tau].\]
 Given $i \in [p]$, we put \[\delta^i: [p - 1] \to [p]\] for the unique order-preserving injection whose image does not contain $i$, and
\[\delta_i : [p] \smallsetminus \{i\} \to [p - 1] \] for the unique map satisfying $\delta_i \circ \delta^i = \mathrm{Id}$. We define \[d_i[\G, \tau] \defeq \Delta_{g,n}(\delta^i)([\G, \tau]),\]
and
\[d_{i,j}[\G, \tau] \defeq \Delta_{g, n}  (\delta^{j} \circ \delta^{\delta_{j}(i)}) [\G, \tau] = d_{\delta_{j}(i)} \left(d_j[\G, \tau]\right). \]
Inductively we can define $d_{i_0, \ldots, i_r}[\G, \tau]$ for any $r \leq p$. The value of $d_{i_0, \ldots, i_r}[\G, \tau]$ only depends on the set $\{i_0, \ldots, i_r\} \subseteq [p]$, so we also use the notation
\[d_S[\G, \tau] \defeq d_{i_0, \ldots, i_r}[\G, \tau]  \]
when $S = \{i_0, \ldots, i_r \}$.
\begin{figure}[h]
	\centering
	\includegraphics[scale=1]{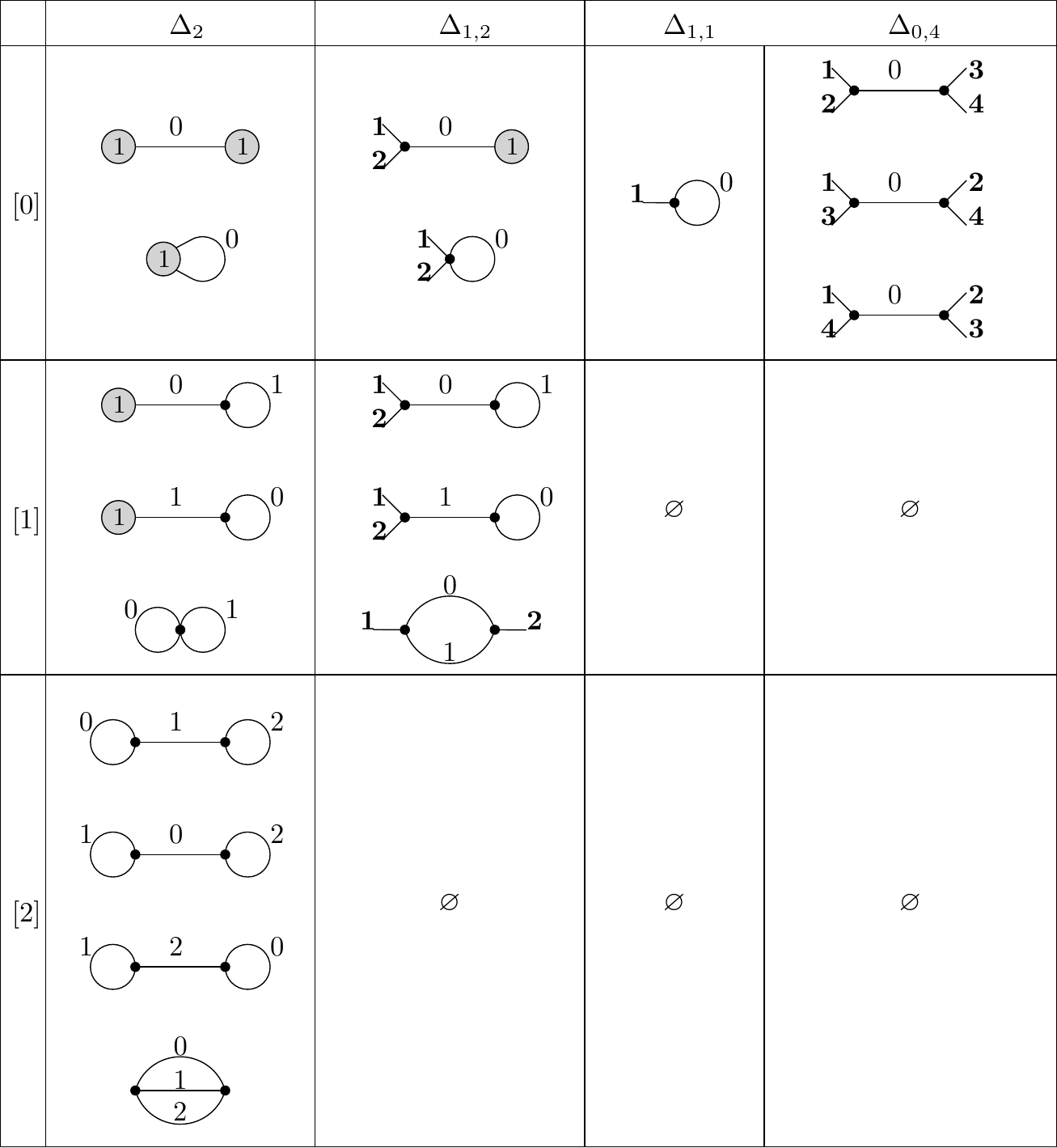}
	\caption{The $\sS_{p + 1}$-sets $\Delta_2[p]$, $\Delta_{1,2}[p]$, $\Delta_{1,1}[p]$, and $\Delta_{0, 4}[p]$ for $0 \leq p \leq 2$. Boldface numbers on half-edges indicate the markings supported at that vertex. Vertices which are not filled are of weight zero.}
	\label{FunctorExample}
\end{figure}
	\begin{figure}[h]
	\centering
	\includegraphics[scale=1]{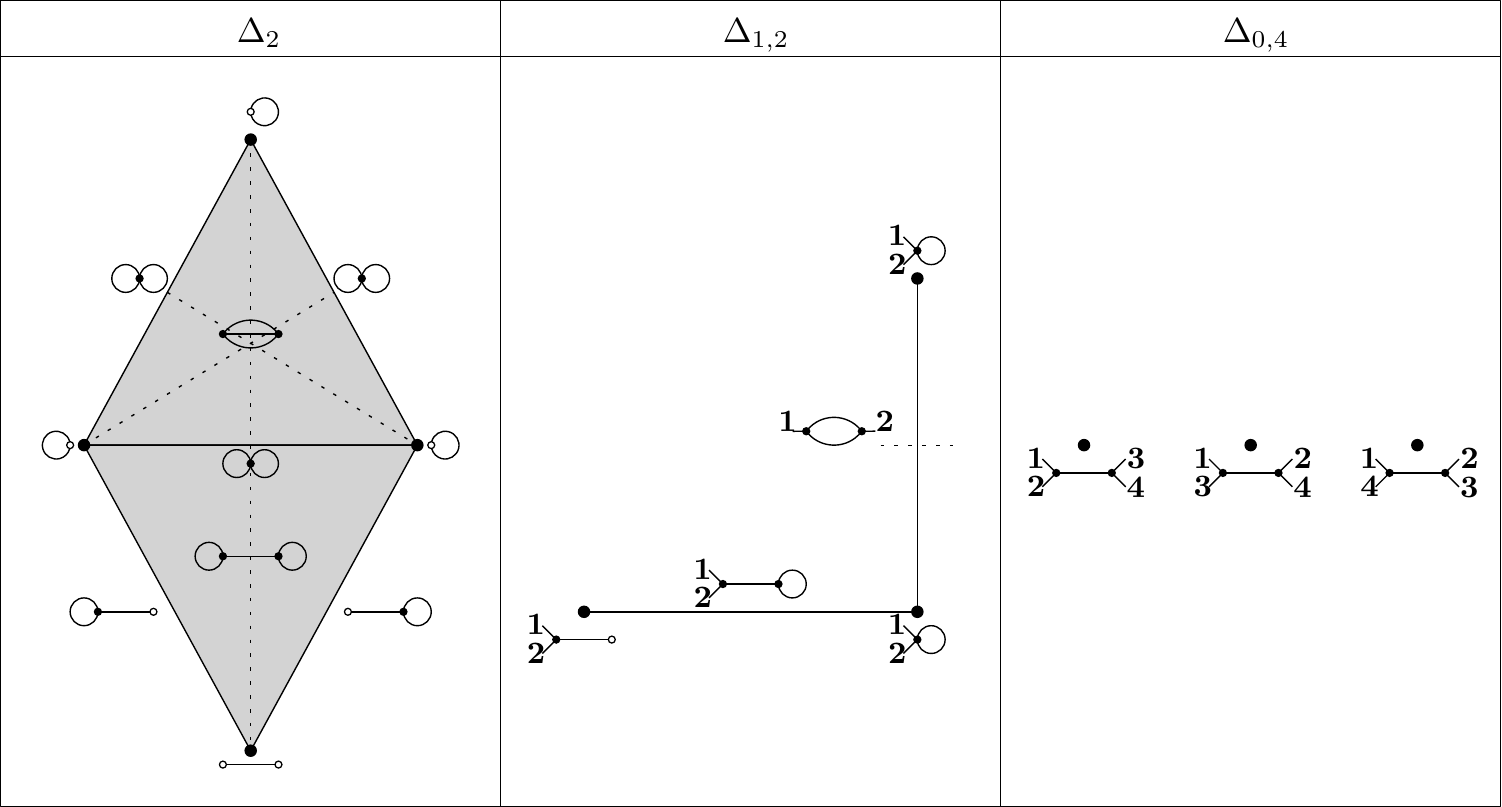}
	\caption{The geometric realizations of $\Delta_{2}$, $\Delta_{1, 2}$, and $\Delta_{0, 4}$, with simplices labelled by the graphs they represent. Dashes indicate lines of symmetry for the self-gluing of simplices. Vertices of weight one are hollow.}
	\label{GeomExample}
\end{figure}
\begin{rem}
	Note that the functor $\Delta_{0,3}$ takes a nonempty value only on $[-1]$, as $\Gamma_{0, 3}$ has a unique object up to isomorphism, consisting of a vertex of weight $0$ supporting all three markings. This is consistent with the fact that $\Delta_{0,3}$ is the link of the generalized cone complex $M_{0,3}^\trop$ consisting of a single point. Throughout the paper we assume $3g - 3 + n > 0$ so as to exclude this trivial case.
\end{rem}	
\subsection{Automorphisms of symmetric $\Delta$-complexes}
An automorphism of $\Delta_{g, n}$ is given by a natural isomorphism of functors $\Phi: \Delta_{g, n} \to \Delta_{g, n}$. This is the data of a bijection $\Phi_p: \Delta_{g, n}[p] \to \Delta_{g, n}[p]$ for each $p \geq -1$, respecting both the simplicial and symmetric structure. That is, for any $i \in [p]$, we must have a commuting diagram
\begin{equation}\label{SimplicialStructure}
\begin{tikzcd}
&\Delta_{g, n}\lbrack p \rbrack \arrow[r, "\Phi_p"] \arrow[d, "d_i"] & \Delta_{g, n}\lbrack p \rbrack \arrow[d, "d_i"]\\
&\Delta_{g, n}\lbrack p-1 \rbrack \arrow[r, "\Phi_{p - 1}"] & \Delta_{g, n}\lbrack p-1 \rbrack
\end{tikzcd},
\end{equation}
and for any $\mathfrak{a} \in \sS_{p + 1}$, a commuting diagram
\begin{equation}\label{SymmetricStructure}
\begin{tikzcd}
&\Delta_{g, n}\lbrack p \rbrack \arrow[r, "\Phi_p"] \arrow[d, "\mathfrak{a}^*"] & \Delta_{g, n}\lbrack p \rbrack \arrow[d, "\mathfrak{a}^*"]\\
&\Delta_{g, n}\lbrack p \rbrack \arrow[r, "\Phi_{p}"] & \Delta_{g, n}\lbrack p \rbrack
\end{tikzcd};
\end{equation}
that is, we must have \[d_i\Phi_p[\G, \tau] = \Phi_{p - 1} d_i[\G, \tau] \] and \[\mathfrak{a} \cdot \Phi_p[\G, \tau] = \Phi_p \left(\mathfrak{a} \cdot [\G, \tau]\right) = \Phi_p[\G, \mathfrak{a} \circ \tau].\] We sometimes refer to the above property as ``$\sS$-equivariance", thinking of $\sS$ as the groupoid
\[\sS = \coprod_{p \geq 1} \sS_p. \] Conversely, any $\sS$-equivariant collection of bijections $\Phi = \{\Phi_p \}_{p = -1}^{\infty}$ satisfying (\ref{SimplicialStructure}) for all $p$ determines a natural isomorphism of functors, since any injection $\iota: [p] \to [q]$ can be factored as an appropriate sequence of maps of the form $\delta^i$ followed by an element of $\sS_{q + 1}$. For ease of notation, we will usually suppress the subscript on the natural transformation $\Phi$ and write $\Phi[\G, \tau]$ instead of $\Phi_p[\G, \tau]$. Note that for any $p$-simplex $[\G, \tau]$, we have
\[\Stab_{\sS_{p + 1}}[\G, \tau] \cong \Aut_{E}(\G), \]
where we define $\Aut_{E}(\G)$ to be the quotient of $\Aut(\G)$ by the normal subgroup of automorphisms which act trivially on $E(\G)$. By $\sS$-equivariance, we see that
\begin{equation}\label{Stabilizers}
\Stab_{\sS_{p + 1}}[\G, \tau] = \Stab_{\sS_{p + 1}}\Phi[\G, \tau]
\end{equation}
If we put $\Phi[\G, \tau] = [\G', \tau']$ then the edge-labellings $\tau, \tau'$ induce faithful actions of $\Aut_E(\G)$ and $\Aut_E(\G')$ on the set $[p]$ and thus two embeddings
\[\Pi: \Aut_E(\G) \to \sS_{p + 1} \]
and
\[\Pi': \Aut_E(\G') \to \sS_{p + 1}. \]
 Equation \ref{Stabilizers} is equivalent to the statement that the images of these maps agree, as subgroups of $\sS_{p + 1}$. 
 \begin{notn}\label{PhiG}
 	Observe that because of $\sS$-equivariance, whenever we are given a graph $\G \in \mathrm{Ob}(\mathsf{\Gamma}_{g, n})$ and $\Phi \in \Aut(\Delta_{g, n})$, we get a unique graph $\Phi \G \in \mathrm{Ob}(\mathsf{\Gamma}_{g, n})$ by picking any edge-labelling $\tau$ of $\G$, and then setting $\Phi \G = \G'$ where $\Phi[\G, \tau] = [\G', \tau']$. Any other choice of edge-labelling $\pi$ of $\G$ is related to $\tau$ by an element $\mathfrak{a} \in \sS_{p + 1}$, so we have
 	\[\Phi[\G, \pi] = \mathfrak{a} \cdot [\G', \tau'] = [\G', \mathfrak{a} \circ \tau']. \]
 	In particular, when we know the action of $\Phi$ on $[\G, \tau]$ for some choice of $\tau$, we know the action for all possible choices of $\tau$. As such, another way to think about $\Phi \in \Aut(\Delta_{g, n})$ is as the data of permutations of $\mathsf{\Gamma}_{g, n}(p)$ for all $p \geq -1$, together with bijections $\Phi_{\G} : E(\G) \to E(\Phi \G)$ for every $\G \in \mathrm{Ob}(\mathsf{\Gamma}_{g, n})$, which induce isomorphisms $\Aut_{E}(\G) \cong \Aut_{E}(\Phi \G)$ making $\Phi_{\G}$ equivariant. Moreover, for any $e \in E(\G)$, we must have $\Phi \G/ \Phi_\G (e) = \Phi(\G/e)$, and the maps $\Phi_{\G}, \Phi_{\G/e}$ must fit into a commutative diagram
 	\[\begin{tikzcd}
 	&E(\G) \arrow[r, "\Phi_\G"] &E(\Phi\G) \\
 	&E(\G/e) \arrow[r, "\Phi_{\G/e}"] \arrow[u, "c_e^*"] &E(\Phi\G/ \Phi_\G (e) ) \arrow[u, "c_{\Phi_\G(e)}^*"]
 	\end{tikzcd}. \]
 	Hereafter we will use the notation $\Phi[\G, \tau] = [\Phi\G, \Phi\tau]$, with the understanding that $\Phi\tau$ is the edge-labelling of $\Phi \G$ determined by $\tau$ and the bijection $\Phi_{\G}: E(\G) \to E(\Phi\G)$.
 \end{notn}	
 
 \begin{exmp}\label{SmallExample}
  We now show how the property of $\sS$-equivariance allows us to compute $\Aut(\Delta_{g, n})$ in some small examples, namely for those $(g, n)$ such that $3g - 3 + n > 0$ but such that $2g - 2 + n < 3$. These $(g, n)$ are given by $(0, 4)$, $(1, 1)$, $(1, 2)$, and $(2, 0)$. See Figure \ref{FunctorExample} for a complete list of the positive dimensional simplices of $\Delta_{g, n}$ in these cases. We see from Figure \ref{FunctorExample} that $\Delta_{1,1}$ consists of a single vertex, hence $\Aut(\Delta_{1,1})$ is trivial, while $\Delta_{0, 4}$ consists of three disjoint vertices, hence $\Aut(\Delta_{0, 4}) \cong S_3$. When $(g, n) = (2, 0)$, the four $2$-simplices are completely distinguished by their $\sS_3$-stabilizers. This forces any automorphism of $\Delta_2$ to fix each top-dimensional simplex. Since any other simplex of $\Delta_2$ is the face of some $2$-simplex (in general, it is well-known that $\Delta_{g, n}$ is always \textsl{pure}), we may conclude that $\Aut(\Delta_2)$ is trivial.
 	
 	We now turn to $\Aut(\Delta_{1,2})$, which again has the property that every simplex is contained in a top-dimensional face, where $\dim \Delta_{1, 2} = 1$. There is a unique $1$-simplex which has a non-trivial $\sS_2$-stabilizer, and hence that simplex is fixed by any automorphism of $\Delta_{1,2}$. The vertex in $\Delta_{1,2}$ corresponding to a graph with a loop is a face of the edge which has a nontrivial stabilizer, and thus must also be fixed by any automorphism of $\Delta_{1,2}$. There is only one other vertex, which therefore is also fixed by any automorphism. This forces any automorphism to fix the remaining two edges of $\Delta_{1,2}$, since switching them would force the switching of the vertices, and we conclude that $\Aut(\Delta_{1, 2})$ is trivial as well.
 \end{exmp}
After Example \ref{SmallExample}, the remaining cases of Theorem \ref{Main} occur when $g, n \geq 0$ satisfy $2g - 2 + n \geq 3$.

\subsection{The $S_n$-action on $\Delta_{g, n}$.} Suppose $n \geq 1$. Given a $\Gamma_{g, n}$-object $\G = (G, w, m)$ and $\sigma \in S_n$, define
\[\sigma\G \defeq (G, w, m \circ \sigma^{-1}). \]
We have identifications $E(\sigma\G) = E(\G)$ and $V(\sigma\G) = V(\G)$, so the formula
\[\sigma \cdot [\G, \tau] = [\sigma\G, \tau] \]
makes sense, and defines an $S_n$-action on $\Delta_{g, n}[p]$ for all $p \geq -1$. This action is compatible with the $\sS$-action (\ref{SymmetricStructure}) and all boundary maps (\ref{SimplicialStructure}), and thus defines an $S_n$-action on $\Delta_{g,n}$ by automorphisms. In this way we have a map
\[S_n \to \Aut(\Delta_{g, n}), \]
and we put $f_\sigma \in \Aut(\Delta_{g,n})$ for the image of $\sigma \in S_n$ under this map. As discussed in the introduction, this action of $S_n$ on $\Delta_{g, n}$ is functorially induced by the $S_n$-action on $\overbar{\mathcal{M}}_{g, n}$ by toroidal automorphisms.

\section{First properties of $\Aut(\Delta_{g, n})$}\label{RestrictionExistence}
Continue to assume that $g, n \geq 0$ and $3g - 3 + n > 0$. In this section, we show that $\Aut(\Delta_{g, n})$ preserves the number of vertices of a graph: for any $\G \in \mathrm{Ob}(\mathsf{\Gamma}_{g, n})$ and $\Phi \in \Aut(\Delta_{g, n})$, we have $|V(\G)| = |V(\Phi\G)|$, where $\Phi\G$ is as determined in Notation \ref{PhiG}.
\subsection{The subcomplexes $\V^{i}_{g, n} \subseteq \Delta_{g, n}$ for $i \geq 1$.}
We can filter $\Delta_{g, n}$ by subcomplexes
\[\V^{1}_{g, n} \subset \V^{2}_{g, n} \subset \cdots \subset \V^{2g - 2 +n}_{g, n} = \Delta_{g, n}, \]
where $\V^{i}_{g, n}$ parameterizes those graphs with at most $i$ vertices. To be precise, a \textsl{\textbf{subcomplex}} of a symmetric $\Delta$-complex is a subfunctor. The subcomplex $\V^{i}_{g, n}$ corresponds to the subfunctor
\[\V_{g,n}^i[p] = \{[\G, \tau] \in \Delta_{g,n}[p] \mid |V(\G)| \leq i \}. \]
As defined, $\V_{g, n}^i$ is a subfunctor of $\Delta_{g, n}$ because edge-contractions either preserve or lower the number of vertices of a graph. We shall now compute the dimensions of the subcomplexes $\V^{i}_{g, n}$ and show that they are \textsl{pure}, in the sense of the following definition.
\begin{defn}
	Given a symmetric $\Delta$-complex $X: \mathrm{I}^\op \to \cat{Set}$, we say $\xi \in X[p]$ is a \textsl{\textbf{facet}} if $\xi$ is not in the image of $X(\delta^i)$ for any $i \in [p + 1]$. We say $X$ is \textsl{\textbf{pure of dimension $d$}} if all of the facets of $X$ are of dimension $d$, i.e. contained in $X[d]$.
\end{defn}	
\begin{prop}\label{PureDimensional}
	For each $i$ such that $1 \leq i \leq 2g - 2 + n$, the subcomplex $\V^{i}_{g, n}$ is pure of dimension $g + i - 2$. Moreover, we have the equality
	\[\V_{g, n}^{2g - 2 + n} = \Delta_{g, n}. \]
\end{prop}	
\begin{proof}
	The purity and dimension of $\Delta_{g,n}$ are well-known, but we include the argument here for the sake of completeness. If a stable graph $\G$ contains a vertex $v \in \G$ such that either $w(v) > 0$ or $\val(v) + |m^{-1}(v)| > 3$, then there exists a stable graph $\hat{\G}$ such that $|E(\hat{\G})| > |E(\G)|$ and a collection of edges $S \subseteq E(\hat{\G})$ such that $\hat{\G}/S \cong \G$. Thus, if $\G$ is maximal with respect to the poset determined by edge-contractions, every vertex $v \in V(\G)$ satisfies $w(v) = 0$ and $\val(v) + |m^{-1}(v)| = 3$. We must have $b^1(\G) = g$, so
	\begin{equation}\label{GenusBettiBaby}
	|E(\G)| - |V(\G)| + 1 = g
	\end{equation}
	Since
	\[3 |V(\G)| = \sum_{v \in V(\G)} (\val(v) + |m^{-1}(v)|) = |H(\G)| + n, \]
	and $|H(\G)| = 2|E(\G)|$, we conclude that
	\[|E(\G)| = 3g - 3 + n \quad \mbox{and} \quad |V(\G)| = 2g - 2 + n.\]
	Thus we have proven that if $[\G, \tau]$ is a facet of $\Delta_{g, n}$, then $[\G, \tau] \in \V^{2g - 2 + n}_{g, n}[3g - 4 + n]$. This simultaneously proves that $\V^{2g - 2 + n}_{g, n} = \Delta_{g, n}$ and that both complexes are pure of dimension $3g - 4 + n$. Now, if $[\G, \tau] \in \V^{i}_{g, n}$ is a facet, then $\G$ has $i$ vertices and all vertex weights equal to zero, i.e. $b^1(\G) = g$. Thus we conclude that $|E(\G)| = g + i - 1$, which shows that $\V^{i}_{g, n}$ is pure of dimension $g + i - 2$.
\end{proof}	
\begin{prop}\label{VerticesPreserved}
	Given $\Phi \in \Aut(\Delta_{g, n})$ and $[\G, \tau] \in \V^{i}_{g, n}[p] \smallsetminus \V^{i - 1}_{g, n}[p]$, we have $\Phi[\G, \tau] \in  \V^{i}_{g, n}[p] \smallsetminus \V^{i - 1}_{g, n}[p]$. Therefore we have a map of groups
	\[\rho^i_{g, n}: \Aut(\Delta_{g,n}) \to \Aut(\V^i_{g, n}) \]
	given by restriction.
\end{prop}
Write $[\Phi\G, \Phi\tau] = \Phi[\G, \tau]$. As an automorphism of a symmetric $\Delta$-complex takes $p$-simplices to $p$-simplices, we certainly have $[\Phi\G, \Phi\tau] \in \Delta_{g,n}[p]$. Thus, we have \[[\Phi\G, \Phi\tau] \in \V^{i}_{g, n}[p] \smallsetminus \V^{i - 1}_{g, n}[p] \] if and only if
\[b^1(\Phi\G) = b^1(\G). \]
As such, Proposition \ref{VerticesPreserved} may be reformulated as follows.
\begin{prop}\label{GenusPreserved}
Given $\Phi \in \Aut(\Delta_{g, n})$ and $[\G, \tau] \in \Delta_{g, n}[p]$, then \[b^1(\Phi\G) = b^1(\G). \]
\end{prop}
To prove Proposition \ref{GenusPreserved}, we study those graphs in $\Gamma_{g, n}$ with one vertex. For $k \geq 0$, fix a graph $R^k$ with one vertex and $k$ loops. For any $n \geq 0$, there is a unique marking function $m^{k}_n: I_n \to V(R^k)$. For fixed $g ,n \geq 0$ and $0 \leq k \leq g$, we put
\[\R^k_{g, n} = (R^k, m^k_n, w^k_g), \]
where $w^k_g(v) = g - k$ for the unique vertex of $R^k$. Then $\R^k_{g, n}$ defines a $\Gamma_{g, n}$-object for all $(g, n)$ and $k$ with $3g - 3 + n > 0$ and $0 \leq k \leq g$. Since all edge-labellings of $\R^k_{g, n}$ lead to isomorphic objects of $\Gamma_{g, n}^\EL$, we have named a unique $(k - 1)$-simplex $[\R^k_{g, n}] \in \Delta_{g,n}[k - 1]$ for all $g, n, k$ in this range. Note that $\R^0_{g, n}$ is the final object of $\Gamma_{g, n}$ and $[\R^{0}_{g, n}]$ is the unique element of $\Delta_{g, n}[-1]$, so any $\Phi \in \Aut(\Delta_{g,n})$ preserves $[\R^{0}_{g, n}]$. 

We now show that when $g \geq 1$, the $0$-simplex $[\R^{1}_{g, n}]$ must also be fixed by $\Aut(\Delta_{g, n})$. For the proof, we require the notion of a \textsl{\textbf{bridge}}, which is an edge of a graph such that deleting it disconnects the graph.
\begin{lem}\label{NonsepPreserved}
	Suppose	$3g - 3 + n > 0$ and $g \geq 1$. Let $\Phi \in \Aut(\Delta_{g,n})$. Then $\Phi[\R_{g, n}^1] = [\R_{g, n}^1]$.
\end{lem} 
\begin{proof}
We shall prove that $[\R^1_{g, n}]$ is the unique vertex of $\Delta_{g,n}$ which is a face of every single top-dimensional simplex. In fact $[\R^1_{g, n}]$ is a face of any simplex $[\G, \tau] \in \Delta_{g, n}[p]$ with $b^1(\G) \geq 1$: we can make a contraction $\G \to \R^1_{g, n}$ by contracting all edges of $\G$ besides some fixed edge which is contained in a cycle. Thus $[\R^1_{g, n}]$ is a face of all top dimensional simplices of $\Delta_{g,n}$. To show that it is the only vertex with this property, it suffices to exhibit a top-dimensional simplex of $\Delta_{g,n}$ whose only $0$-dimensional face is $[\R^1_{g, n}]$. For this it suffices to exhibit a $\Gamma_{g,n}$-object $\G$  which has no bridges, and such that all the vertices $v\in V(\G)$ have $\val(v) + |m^{-1}(v)| = 3$. When $g = 1$, we can take $\G$ to be an $n$-cycle, such that each vertex of $\G$ supports one marking. When $g \geq 2$, we take the $\Gamma_g$-object $\G$ of Figure \ref{Bridgeless}, and choose some distribution of $n$ additional vertices, each supporting a unique marking.
\begin{figure}[h]
	\centering
	\includegraphics[scale=1]{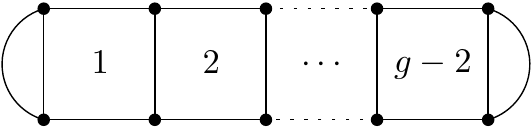}
	\caption{A maximal $\Gamma_g$-object with $3g - 3$ edges and no bridges.}
	\label{Bridgeless}
\end{figure}	
\end{proof}
As a corollary of Lemma \ref{NonsepPreserved}, we have that any $\Phi \in \Aut(\Delta_{g,n})$ preserves bridges.
\begin{cor}\label{BridgesPreserved}
Let $\Phi \in \Aut(\Delta_{g,n})$, and suppose $[\G, \tau] \in \Delta_{g,n}[p]$. Then $\tau^{-1}(i) \in E(\G)$ is a bridge of $\G$ if and only if $(\Phi\tau)^{-1}(i) \in E(\Phi\G)$ is a bridge of $\Phi \G$.
\end{cor}
\begin{proof}
	Let $\mathcal{B} \subseteq [p]$ be the set of indices corresponding to bridges of $\G$. Then
	\[\mathcal{B} = \{k \in [p] \mid  d_{[p] \smallsetminus \{k\}}[\G, \tau] \neq [\R^1_{g, n}]\}. \]
	The claim is now immediate from Lemma \ref{NonsepPreserved}.
\end{proof}
We require one final observation before the proof of Proposition \ref{GenusPreserved}.
\begin{lem}
	Suppose $3g - 3 + n > 0$ and $g \geq 1$. Then there exists a $\Gamma_{g, n}$-object $\G$ with $3g - 3 + n$ edges, such that
	\begin{itemize}
	 \item $b^1(\G) = g$, and
	 \item every edge of $\G$ is either a loop or a bridge.
	\end{itemize}
\end{lem}
\begin{proof}
	When $(g, n) = (1, 1), (1, 2), (2, 0)$, this is seen in Figure \ref{FunctorExample}. It is easy to construct such graphs in general: Figure \ref{TopDims} finishes the proof.
	\begin{figure}[h]
		\centering
		\includegraphics[scale=1]{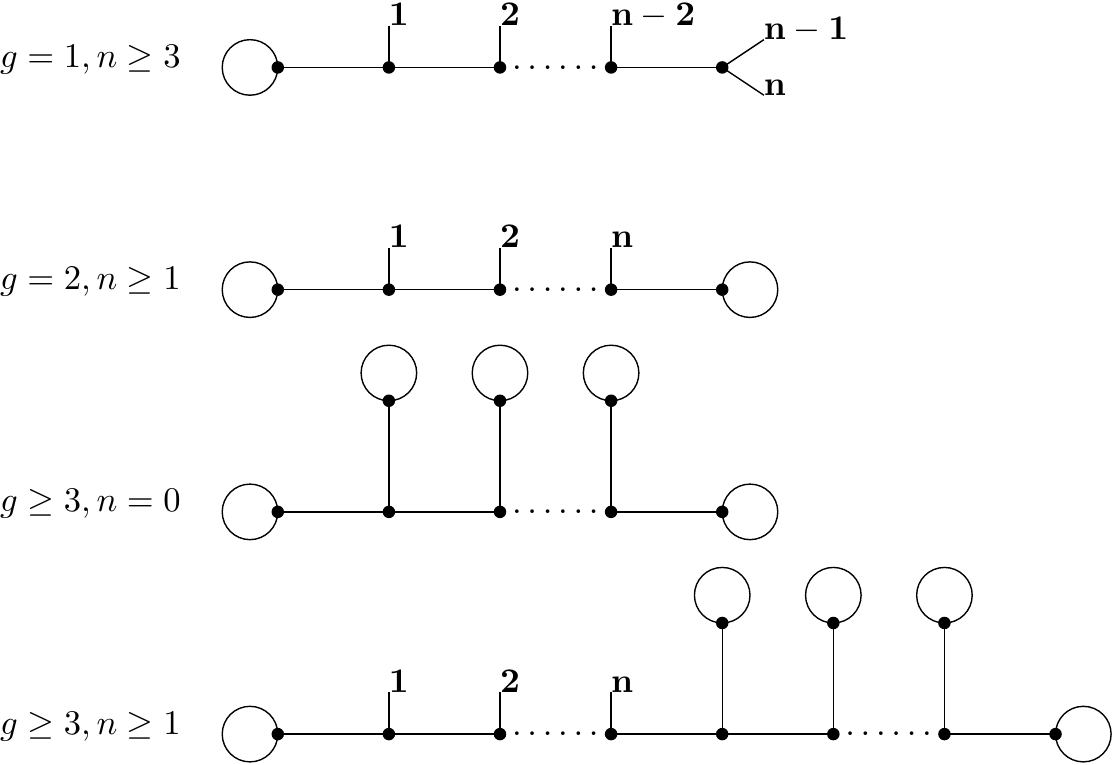}
		\caption{Maximal graphs in $\Gamma_{g, n}$ with only bridges or loops as edges.}
		\label{TopDims}
	\end{figure}	
\end{proof}	
We are now ready to prove Proposition \ref{GenusPreserved}.
\begin{proof}[Proof of Proposition \ref{GenusPreserved}]
When $g = 0$, we have $b^1(\G) = 0$ for all graphs of $\Gamma_{0, n}$, so the proposition holds trivially in this case. When $g \geq 1$, observe that a $\Gamma_{g,n}$-object $\G$ satisfies $b^1(\G) \geq k$ if and only if there exists a morphism $\G \to \R^{k}_{g, n}$. Therefore it suffices to prove that 
\[\Phi[\R_{g, n}^k] = [\R_{g, n}^k] \]
for all $1 \leq k \leq g$: if $\Phi$ fixes the simplex $[\R_{g, n}^k]$, then it must fix setwise its \textsl{simplicial star}, i.e. the set of all simplices having it as a face. Moreover, if
\[\Phi[\R_{g, n}^g] = [\R_{g, n}^g], \]
then $\Phi[\R_{g, n}^k] = [\R_{g, n}^k]$ for all $1 \leq k \leq g$, since $[\R_{g, n}^k]$ is a face of $[\R_{g, n}^g]$. Pick a stable graph $\G$ such that $\G$ has $3g - 3 + n$ edges, all of which are either loops or bridges. Then $b^1(\G) = g$ since $\G$ is maximal, so $\G$ has precisely $g$ loops and $2g - 3 +n$ bridges. Let $\tau : E(\G) \to [3g - 4 + n]$ be any edge-labelling, and let $\mathcal{B} \subseteq [p]$ be the set of indices of bridges in $\G$. Then $\mathcal{B}$ is also the set of indices of bridges in $\Phi \G$ by Corollary \ref{BridgesPreserved}, so we have
 \[d_{\mathcal{B}}[\G, \tau] = d_{\mathcal{B}}[\Phi\G, \Phi\tau] = [\R_{g, n}^g].\]
Since
\[d_{\mathcal{B}}[\Phi\G, \Phi\tau] = \Phi(d_{\mathcal{B}}[\G, \tau]) = \Phi[\R_{g, n}^g], \]
the proof is complete.
\end{proof}	
Proposition \ref{GenusPreserved} implies Proposition \ref{VerticesPreserved}, so we get a restriction homomorphism \[\rho^{i}_{g, n}: \Aut(\Delta_{g, n}) \to \Aut(\V^{i}_{g, n})\] for all $i \geq 1$. In the next section we will show that this map is an embedding when $i = 2$.
Before proceeding, we record an important corollary of Proposition \ref{VerticesPreserved} which we will use several times throughout the paper. We adopt the convention that a $1$-cycle is a loop, while a $2$-cycle is given by a pair of parallel nonloop edges.
\begin{cor}\label{CyclesPreserved}
Let $[\G, \tau] \in \Delta_{g, n}[p]$ and $\Phi \in \Aut(\Delta_{g, n})$. Then, for all $k \geq 1$, a subset $S \subseteq [p]$ with $|S| = k$ corresponds via $\tau$ to a $k$-cycle in $\G$ if and only if it corresponds via $\Phi\tau$ to a $k$-cycle in $\Phi \G$; that is, the bijection 
\[\Phi_{\G}: E(\G) \to E(\Phi \G) \]
of Notation \ref{PhiG} preserves all cycles.
\end{cor}	
\begin{proof}
	Suppose $|V(\G)| = N$. The proof is by induction on $k$.  When $k = 1$, the claim is that an index $\ell \in [p]$ labels a loop in $\G$ if and only if $\ell$ labels a loop in $\Phi \G$. Observe that $\ell$ is a loop index if and only if $d_\ell[\G, \tau]$ is in $\V^{N}_{g, n}$, so the claim is a direct consequence of Proposition \ref{VerticesPreserved}. The inductive step follows from the fact that $\Phi d_j = d_j \Phi$ for all $j \in [p]$, together with the observation that, for $k \geq 2$, a subset $C \subseteq E(\G)$ forms a $k$-cycle in $\G$ if and only if $c_e(C \smallsetminus \{e\} ) \subseteq E(\G/e)$ forms a $(k - 1)$-cycle of $\G/e$ for all $e \in C$.
\end{proof}	

\section{A reconstruction algorithm for edge-labelled stable graphs}\label{RestrictionInjects}
In this section we prove Theorem \ref{TwoVertex}, which states that the restriction map
\[\rho^{2}_{g, n}: \Aut(\Delta_{g, n}) \to \Aut(\V^{2}_{g, n}) \]
is an injection. The proof relies on an explicit procedure for reconstructing top-dimensional simplices $[\G, \tau]$ of $\V^{i}_{g, n}$ from their ordered list of faces which lie in $\V^{i - 1}_{g, n}$, when $i \geq 3$. Graph-theoretically, our proof amounts to an algorithm for reconstructing an edge-labelled stable pair $(\G, \tau)$ satisfying $b^1(\G) = g$ and $|V(\G)| \geq 3$ from its \textsl{nonloop contraction deck}, as in Definition \ref{GlobalInvariants} below. First, we set some notation.

Given a $\Gamma_{g, n}^{\EL}$-object $(\G, \tau : E(\G) \to [p])$ and $j \in [p]$, we set $e_j\defeq \tau^{-1}j$ and put $c_j : \G \to \G/e_j$ for the contraction of edge $e_j$. We let $\tau_j: E(\G/e_j) \to [p - 1]$ be the unique edge-labelling fitting into the commutative square
\[\begin{tikzcd}
&E(\G) \arrow[r, "\tau"] & \lbrack p \rbrack\\
&E(\G/e_j) \arrow[r, "\tau_j"] \arrow[u, "c_j^*"] &\lbrack p - 1 \rbrack \arrow[u, "\delta^j"]
\end{tikzcd}. \] That is, $c_j = c_{\{j\}}$ and $\tau_j = \tau_{\delta^j}$ in the notation of Diagram \ref{FaceMaps}, so in simplicial terms we have $d_j[\G, \tau] = [\G/e_j, \tau_j]$. We now define what is, for the purpose of this section, the fundamental invariant of a pair $(\G, \tau)$ in $\Gamma_{g, n}^{\EL}$.
\begin{defn}[Nonloop contraction deck]\label{GlobalInvariants}
	Let $(\G, \tau: E(\G) \to [p])$ be an object of $\Gamma_{g, n}^{\mathrm{EL}}$. The \textsl{\textbf{nonloop contraction deck}} of $(\G, \tau)$, denoted $\D^\G_\tau$, is an ordered list of $\Gamma_{g, n}^{\EL}$-objects:
	\[\D^\G_\tau = \{ ((\G/e_j, \tau_j), j) \mid e_j \text{ is not a loop of }\G \}. \]
	Thus $\D^\G_{\tau}$ is the list of nonloop contractions of $(\G, \tau)$, where elements are indexed according to the edge which was contracted.
\end{defn}

 Given two lists $\D_1 = \{(\G_i, i)\}_{i \in J_1}$,  $\D_2 = \{(\H_i, i)\}_{i \in J_2}$ of $\Gamma_{g, n}^{\EL}$-objects indexed by subsets $J_1, J_2 \subseteq [p]$, we write
\[\D_1 \cong \D_2 \]
if and only if:
\begin{itemize}
	\item $J_1 = J_2$, and 
	\item $\G_i \cong \H_i$ for all $i \in J_1$.
\end{itemize}
The proof of Theorem \ref{TwoVertex} relies on the main technical result of this section: up to isomorphism, an edge-labelled pair $(\G, \tau)$ is determined uniquely by its nonloop contraction deck $\D^{\G}_\tau$, assuming that $b^1(\G) = g$ and $|V(\G)| \geq 3$.
\begin{thm}\label{Reconstruct}
	Suppose $(\G, \tau)$ is an edge-labelled pair in $\Gamma_{g, n}^{\EL}$ satisfying $b^1(\G) = g$ and $|V(\G)| \geq 3$. If another edge-labelled pair $(\G', \tau')$ satisfies
	\[\mathcal{D}^{\G'}_{\tau'} \cong \mathcal{D}^\G_{\tau},\]
	then there exists an isomorphism of pairs $(\G, \tau) \cong (\G', \tau')$.
\end{thm}

Theorem \ref{Reconstruct} is the answer to an easier version of a question with some history in graph theory: given an (unlabelled) graph $G$, when does the deck of all one-edge contractions determine $G$ up to isomorphism? The still open \textsl{contraction reconstruction conjecture} posits that this is always the case for simple graphs $G$ with at least four edges. For a survey of this conjecture and partial results, see the PhD thesis of Antoine Poirier ~\cite{Poirier}.
	
We now indicate how Theorem \ref{TwoVertex} is proven, using the above Theorem \ref{Reconstruct}. 
\begin{proof}[Proof of Theorem \ref{TwoVertex}]
	Suppose given an automorphism $\Phi \in \Aut(\Delta_{g,n})$ such that
	\[\Phi|_{\V^2_{g, n}} = \mathrm{Id}|_{\V^2_{g, n}}, \]
	in order to show that $\Phi = \mathrm{Id}$. For this it suffices to show that
	\[\Phi|_{\V^{i}_{g, n}} = \mathrm{Id}|_{\V^i_{g, n}} \] for all $i \geq 2$. We prove this claim by induction, the case $i = 2$ being the base case. So, suppose that 
	\[\Phi|_{\V^{i}_{g, n}} = \mathrm{Id}|_{\V^i_{g, n}}\]
	for some $i \geq 2$, in order to show that
	\[\Phi|_{\V^{i + 1}_{g, n}} = \mathrm{Id}|_{\V^{i + 1}_{g, n}}. \]
	Since the subcomplex $\V^{i + 1}_{g, n}$ is pure of dimension $g + i - 1$, it is enough to show that $\Phi[\G, \tau] = [\G, \tau]$ for any $[\G, \tau] \in \V^{i + 1}_{g, n}[g + i - 1]$, i.e. those pairs $(\G, \tau)$ in $\Gamma^{\EL}_{g, n}$ satisfying $b^1(\G) = g$ and $|V(\G)| = i + 1$. Put $\Phi[\G, \tau] = [\Phi \G, \Phi \tau]$. We define
	\[\Phi(\D^\G_{\tau}) = \{((\Phi (\G/e_{j}), \Phi(\tau_j)), j) \mid e_j \text{ is a nonloop edge of }\G \}. \]
	As $\Phi$ preserves loops (Corollary \ref{CyclesPreserved}) and we have $d_j\Phi = \Phi d_j$ for all $j$, we get an equivalence $\Phi(\D^\G_{\tau}) \cong \D^{\Phi\G}_{\Phi\tau}$. Since each $[\G/e_j, \tau_j]$ is a simplex in $\V^{i}_{g, n}$, our inductive hypothesis gives that $\D^{\G}_{\tau} \cong \Phi(\D^{\G}_{\tau})$. Then $(\G, \tau) \cong (\Phi\G, \Phi\tau)$ by Theorem \ref{Reconstruct}, meaning $[\G, \tau] = [\Phi \G, \Phi \tau]$, and the proof is complete.
\end{proof}
The remainder of this section is devoted to the proof of Theorem \ref{Reconstruct}, which will be split into intermediate results. We require two auxiliary invariants $\O^\G_{\tau}$ and $\Q^\G_\tau$ of a pair $(\G, \tau)$; as we will see, in most cases these are ultimately computable from $\D^{\G}_\tau$. 
\begin{defn}[Auxiliary invariants $\O^\G_\tau$, $\Q^{\G}_\tau$]\label{AuxiliaryInvariants}
	Let $(\G, \tau)$ be an object of the category $\Gamma^{\EL}_{g, n}$. 
\begin{enumerate}[\hspace{1cm}]
		\item[$\O^\G_{\tau}$]--- For each $k \geq 1$, define
	\[\binom{[p]}{k} \defeq \{ S \subseteq [p] \mid |S| = k\}. \]
	Then we set
	\[\O^\G_\tau(k) \defeq \left\{S \in \binom{[p]}{k} \mid \{e_i\}_{i \in S} \text{ forms a }k\text{-cycle of }\G \right\},  \] and define
	\[\O^\G_\tau \defeq \bigsqcup_{k \geq 1} \O^\G_\tau(k). \]
	We call $\O^\G_\tau$ the \textsl{\textbf{total cycle set}} of $(\G, \tau)$.
	\item[$\Q^\G_\tau$]--- Make a map
	\[\langle -, -\rangle^\G_{\tau}: ([p] \sqcup I_n) \times ([p] \sqcup I_n) \to \{0, 1, 2\}  \]
	as follows: when $i, j \in [p]$, we set $\langle i, j\rangle^\G_\tau$ to be the number of vertices shared by the edges $e_i$ and $e_j$, adopting the convention that a loop meets itself at one vertex, and a non-loop edge meets itself at two vertices. We extend this rule to $[p] \sqcup I_n$ by treating markings $\alpha \in I_n$ as though they were loops. The \textsl{\textbf{intersection matrix}} $\Q^\G_{\tau}$ of $(\G, \tau)$ is the $(p + 1 + n) \times (p + 1 + n)$ symmetric matrix \[ \Q^\G_{\tau} \defeq 
	\kbordermatrix{
		&\lbrack p \rbrack &{} & I_n \\ 
		\lbrack p \rbrack & \left(\langle i, j \rangle^\G_{\tau}\right)_{i, j \in [p]} &\vrule  & \left(\langle i , x \rangle^\G_{\tau}\right)_{i \in [p], x \in I_n}  \\ \cline{2-4}
		I_n & \left(\langle x , i \rangle^\G_{\tau}\right)_{x \in I_n, i \in [p]} &\vrule & \left(\langle x, y \rangle^\G_{\tau}\right)_{x, y \in I_n}
	};
	\]
	here both $[p]$ and $I_n$ are ordered in the natural way.

	\end{enumerate}
\end{defn}

\begin{notn}
	In general, we use $i$, $j$, $k$, and $\ell$ to stand for elements of $[p]$, and use $x, y, z$ for elements of $I_n$. When we want to refer to arbitrary elements of $[p] \sqcup I_n$ and remain agnostic about whether they correspond to edges or markings, we will use the Greek letters $\alpha$, $\beta$, $\gamma$, and $\epsilon$. 
\end{notn}	
We now state three intermediate results which combine to prove Theorem \ref{Reconstruct}. The first (Proposition \ref{DcomputesO}) we can prove immediately, while the next two (Propositions \ref{DataTransfer} and \ref{Reconstruction}) require more work; we will use their statements to prove Theorem \ref{Reconstruct}, and then conclude this section with their proofs.

\begin{prop}\label{DcomputesO}
	Suppose $(\G, \tau)$ is any object of $\Gamma_{g, n}^{\EL}$. Then $\O^{\G}_\tau$ is uniquely determined by $\D^{\G}_\tau$.
\end{prop}
\begin{proof}
We first observe that $\ell \in [p]$ is a loop index if and only if it does not list any element of $\D^{\G}_\tau$. For $k \geq 2$, we observe that a subset $S \subseteq [p]$ with $|S| = k$ indexes a $k$-cycle of $(\G,\tau)$ if and only if all of the elements of $S$ are nonloop indices, and for each $j \in S$, the set
\[\delta_j(S \smallsetminus \{j\}) \subseteq [p - 1] \]
indexes a $(k - 1)$-cycle of $(\G/e_j, \tau_j)$. 
\end{proof}

The next result gives an algorithm for calculating $\Q^\G_{\tau}$ from $\D^\G_{\tau}$ and $\O^\G_{\tau}$, once we assume that $\G$ has at least three vertices and satisfies $b^1(\G) = g$. 
\begin{prop}\label{DataTransfer}
	Let $(\G, \tau)$ be an object of the category $\Gamma_{g,n}^\EL$ with $|V(\G)| \geq 3$ and $b^1(\G) = g$. Then there is an explicit algorithm for computing $\Q^\G_\tau$ from $\D^\G_{\tau}$ and $\O^\G_{\tau}$. This means that if we have another pair $(\G', \tau')$ with $|V(\G')| \geq 3$, $b^1(\G') = g$, and $\D^{\G}_\tau \cong \D^{\G'}_{\tau'}$, then we have an equality
	\[\Q^{\G}_{\tau} = \Q^{\G'}_{\tau'}. \]
\end{prop}	

The final ingredient of the proof of Theorem \ref{Reconstruct} is the claim that $\D^\G_{\tau}$ and $\Q^\G_{\tau}$ determine the pair $(\G, \tau)$ up to isomorphism. The proof of the following proposition will be constructive, justifying the title of section.
\begin{prop}\label{Reconstruction}
	Suppose $(\G, \tau)$ is an object of $\Gamma_{g, n}^{\EL}$ with $b^1(\G) = g$. Then the invariants $\D^\G_{\tau}$ and $\Q^\G_{\tau}$ determine the pair $(\G, \tau)$ up to isomorphism; that is, if $\D^\G_{\tau} \cong \D^{\G'}_{\tau'}$ and $\Q^\G_{\tau} = \Q^{\G'}_{\tau'}$, then $(\G, \tau) \cong (\G', \tau')$.
\end{prop}

Before proceeding with the proofs of Propositions \ref{DataTransfer} and \ref{Reconstruction}, we pause to prove Theorem \ref{Reconstruct}.
 
\begin{proof}[Proof of Theorem \ref{Reconstruct}]
Suppose given $(\G, \tau), (\G', \tau')$ in $\Gamma_{g,n}^{\EL}$ such that $b^1(\G) = g$, $|V(\G)| \geq 3$, and $\D^{\G}_{\tau} \cong \D^{\G'}_{\tau'}$. Then of course $b^1(\G') = g$ and $|V(\G')| = |V(\G)|$. Proposition \ref{DcomputesO} implies that $\O^{\G}_{\tau} = \O^{\G'}_{\tau'}$, and then Proposition \ref{DataTransfer} gives that $\Q^{\G}_{\tau} = \Q^{\G'}_{\tau'}$. Using Proposition \ref{Reconstruction}, we may conclude that $(\G, \tau) \cong (\G', \tau')$, as desired.
\end{proof}
\begin{rem}
	The reader might expect that in almost all cases, it should be rather easy to determine $(\G, \tau)$ from $\mathcal{D}^{\G}_\tau$. Our algorithmic proof will illustrate that this is essentially the case as long as $|V(\G)|\geq 5$ and $b^1(\G) = g$ (and we expect that the assumption $b^1(\G) = g$ can be dispensed with here). However, when $|V(\G)| \leq 4$, the assumption that $b^1(\G) = g$ and the stability condition defining $\Gamma_{g, n}$ are both necessary; see Figure \ref{dNotEnough}. Most of the work in the proof of Proposition \ref{DataTransfer} is in dealing with graphs of at most four vertices; it is in these cases when our arguments rely strongly on the stability condition.
\end{rem}	

\begin{figure}
	\centering
	\includegraphics[scale=1]{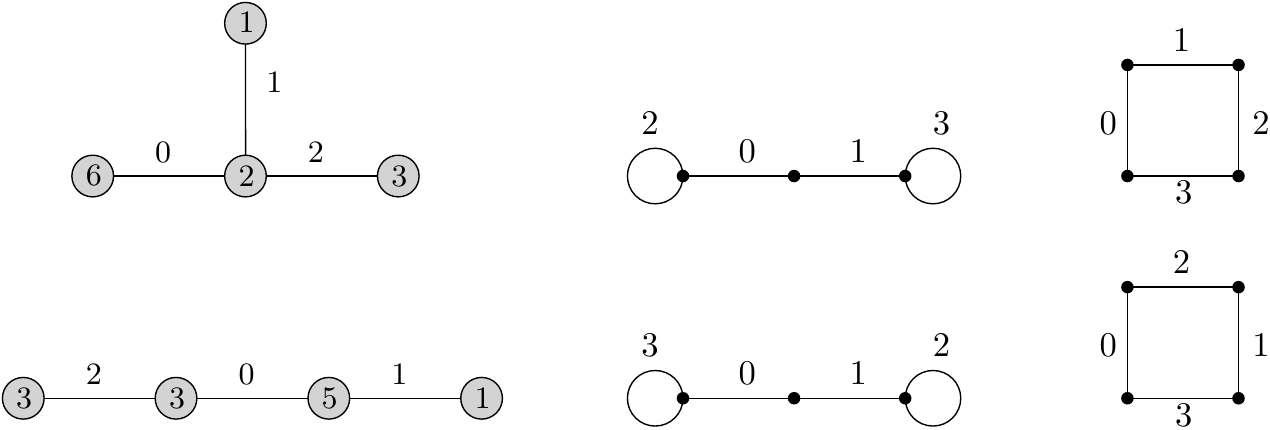}
	\caption{Examples of non-isomorphic pairs of edge-labelled graphs with the same values of $\D^{\G}_\tau$ and $\O^{\G}_\tau$. In the first case, the two edge-labelled graphs are objects of $\Gamma_{12}^{\EL}$, but do not have full genus. The graphs in the second two examples fail stability, so are not objects of any $\Gamma_{g, n}^{\EL}$.}
	\label{dNotEnough}
\end{figure}
We are now tasked with proving Propositions \ref{DataTransfer} and \ref{Reconstruction}. 
\subsection{Proof of Proposition \ref{DataTransfer}}
To prove Proposition \ref{DataTransfer}, we give an algorithm for calculating $\Q^\G_{\tau}$ from $\D^\G_\tau$ and $\O^\G_{\tau}$, given that $b^1(\G) = g$ and $|V(\G)| \geq 3$.

We require three preliminaries. First, we need to name the set of indices corresponding to nonloop edges in an edge-labelled pair $(\G, \tau)$; we put $\mathcal{N}^\G_\tau$ for this set. Clearly $\mathcal{N}^\G_{\tau}$ can be recovered from $\D^\G_\tau$ as the image $\D^\G_\tau \to [p]$ of the map which returns the index of each graph.

Second, recall that for each $j \in [p]$, the map $\delta^j: [p-1] \to [p]$ is the unique order-preserving injection which misses the element $j$, and $\delta_j : [p] \smallsetminus\{j\} \to [p - 1]$ is its inverse. When $n \geq 1$, it will be convenient to extend $\delta_j$ to get \[\delta_j: ([p] \smallsetminus\{j\}) \sqcup I_n \to [p - 1] \sqcup I_n,\] by having it act as the identity on $I_n$. The point is that when $\alpha, \beta \in [p]$, the elements $\delta_j(\alpha), \delta_j(\beta) \in [p - 1]$ correspond via $\tau_j$ to the edges $c_j(e_\alpha)$ and $c_j(e_\beta)$ in $\G/e_j$, but now given a marking $x \in I_n$ it also makes sense to think of $\delta_j(x)$ as the image of $x$ in $\G/e_j$.

Third, we require the notion of a full subgraph.
\begin{defn}\label{FullSubgraph}
	Let $\G = (G, w, m)$ be an object of $\Gamma_{g, n}$. An unweighted, unmarked subgraph $H$ of $G$ is called a \textsl{\textbf{full subgraph}} of $\G$ if for any edge $e$ of $G$ which is not in $H$, we have either that $e$ is a loop or that $e$ is parallel to a nonloop edge which is contained in $H$.
\end{defn}

We now describe the algorithmic proof of Proposition \ref{DataTransfer}. We take as input data $\D^\G_\tau$ and $\O^\G_\tau$, which are known to come from a mystery stable pair $(\G, \tau)$; it is known only that $b^1(\G) = g$ and $|V(\G)| \geq 3$. Our output is the intersection matrix $\Q^\G_\tau$ of $(\G, \tau)$. To follow the logic, the reader may find the directed graph/decision ``tree" in Figure \ref{DecisionTree} helpful.

\begin{enumerate}[\hspace{1cm}]
	\item \textbf{(Step 1)} First, we calculate the diagonal entries of $\Q^\G_{\tau}$; this is easily done using just $\D^\G_{\tau}$, as it is the same as deciding whether a given $\alpha \in [p] \sqcup I_n$ corresponds to a loop, marking, or nonloop edge in $\G$.
	\item \textbf{(Step 2)} We then find all non-diagonal entries of $\Q^\G_{\tau}$ which are equal to $2$ -- these correspond precisely to pairs $(i, j)$ of elements of $[p]$ such that $\{e_i, e_j\}$ forms a $2$-cycle of $\G$. As such, these entries can be read off directly from $\O^\G_\tau$.
	\item \textbf{(Step 3)} We check whether $|V(\G)| \geq 5$. If $|V(\G)| < 5$, then we proceed to Step 4. Otherwise, for all remaining pairs of indices $\alpha, \beta \in [p] \sqcup I_n$ (i.e., those entries of $\Q^\G_\tau$ not yet calculated in Steps 1 and 2), we have an equality 
	\begin{equation}\label{AlmostAllIntersections}
	\langle \alpha, \beta \rangle^\G_{\tau} = \min_{\substack{j \in \mathcal{N}^\G_\tau\\j \neq \alpha, \beta}} \langle \delta_j(\alpha), \delta_j(\beta) \rangle^{\G/e_j}_{\tau_j}.
	\end{equation}
	This means the remaining entries of $\Q^\G_\tau$ are calculated from $\D^\G_\tau$, and the algorithm terminates. The equality (\ref{AlmostAllIntersections}) will be proven in Lemma \ref{GenericCapping}. 
	\item \textbf{(Step 4)} For each $i = 0, \ldots, 6$, we use $\D^\G_\tau$ and $\O^\G_\tau$ to check whether $\G$ has a full subgraph isomorphic to any of the exceptional graphs $\mathbb{E}_i$ of Figure \ref{ExceptionalFamilies}. If it does have such a full subgraph for some $i$, then the remaining entries can be calculated by one of the procedures described in the proof of Lemma \ref{ExceptionalCalculations}, and the algorithm terminates. If $\G$ does not have any such full subgraph, then Lemma \ref{GenericCapping} will imply that the remaining entries of $\Q^\G_\tau$ are given by (\ref{AlmostAllIntersections}) above, so the algorithm also terminates in this case.
\end{enumerate}

\begin{figure}[h]
\centering
\includegraphics[scale=1]{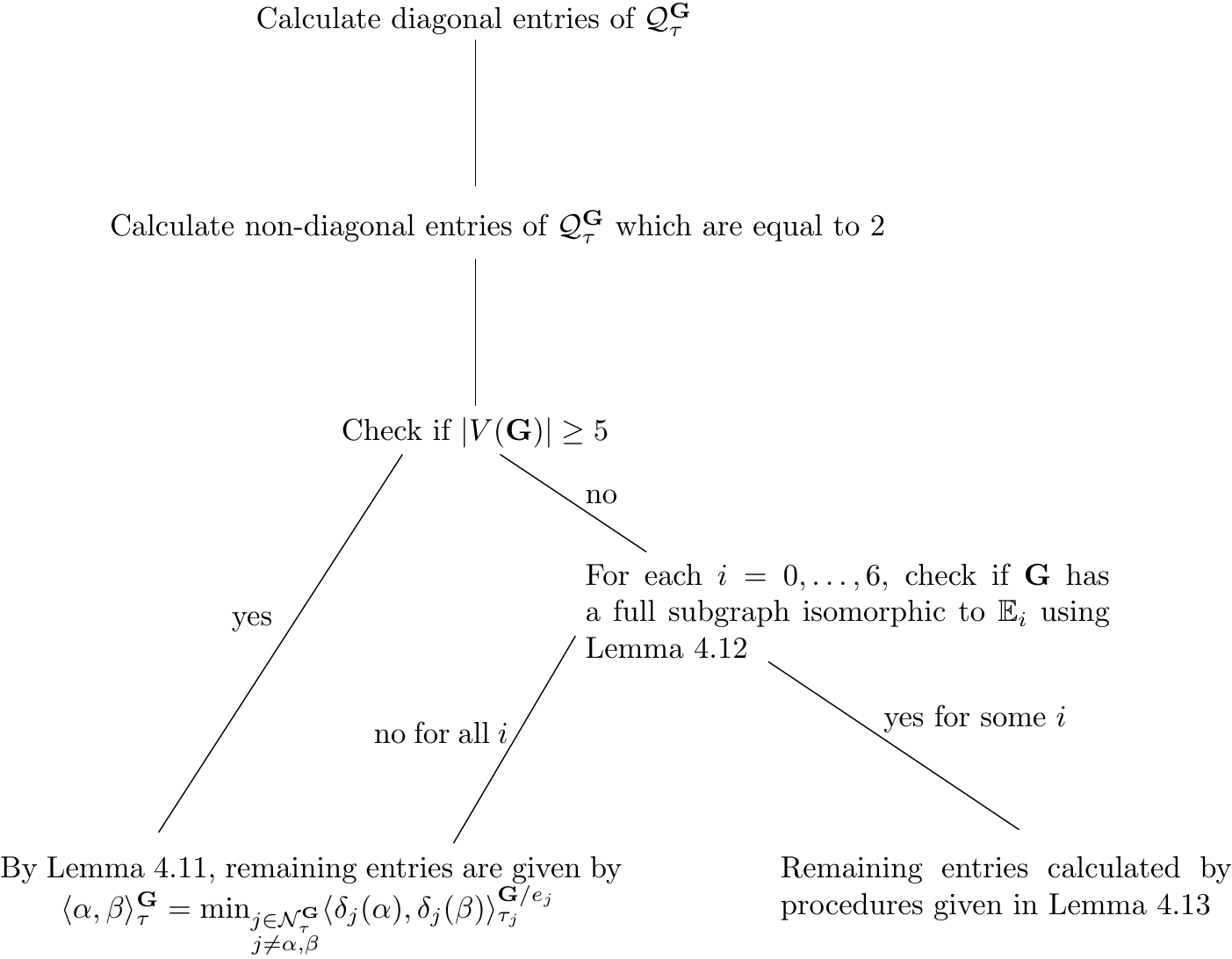}
\caption{A directed graph outlining the algorithmic calculation of $\Q^\G_\tau$ when $b^1(\G) = g$, given input data $\D^\G_\tau$ and $\O^\G_\tau$.}
\label{DecisionTree}
\end{figure}

We now prove that our algorithm terminates, by stating and proving Lemmas \ref{GenericCapping}-\ref{ExceptionalCalculations}. 

\begin{lem}\label{GenericCapping}
	Suppose $(\G, \tau)$ is an edge-labelled pair with $|V(\G)| \geq 3$, and suppose that there exists a pair of indices $\alpha, \beta \in [p] \sqcup I_n$ such that:
	\begin{enumerate}[(i)]
		\item $\alpha$ and $\beta$ do not correspond to a $2$-cycle in $\G$, and
		\item we have
		\[ \langle \alpha, \beta \rangle^\G_{\tau} \neq \min_{\substack{j \in \mathcal{N}^\G_\tau\\j \neq \alpha, \beta}} \langle \delta_j(\alpha), \delta_j(\beta) \rangle^{\G/e_j}_{\tau_j}. \]
	\end{enumerate}
	Then either:
	\begin{enumerate}[(I)]
		\item both $\alpha, \beta \in \mathcal{N}^\G_\tau$, and $\G$ has a full subgraph isomorphic to $\mathbb{E}_i$, for some $i = 0,\ldots,6$, or
		\item exactly one of $\alpha, \beta$ lies in $\mathcal{N}^{\G}_\tau$, and $\G$ has either $\mathbb{E}_5$ or $\mathbb{E}_6$ as a full subgraph.
	\end{enumerate}	
\end{lem}	
\begin{proof}
First observe that if $\alpha, \beta \notin \mathcal{N}^\G_\tau$, then (\ref{AlmostAllIntersections}) always holds: if two loops are on the same vertex of $\G$, then the same is true for their images in all nonloop contractions of $\G$. If they are not on the same vertex of $\G$, then since $|V(\G)| \geq 3$, and $\G$ is connected, there must exist a nonloop edge which does not contain both loops; upon contracting this edge, the two loops are still disjoint. The same argument works if we replace one or both of the loops by markings.

Now suppose that there exist a pair $\alpha, \beta \in \mathcal{N}^\G_\tau$ such that $\langle \alpha, \beta \rangle_\tau^\G = 0$ (i.e., $e_\alpha$ and $e_\beta$ are disjoint nonloop edges of $\G$), and
\[\langle \alpha, \beta \rangle_\tau^\G < \langle \delta_j(\alpha), \delta_j(\beta) \rangle^{\G/e_j}_{\tau_j} \]
for any $j \in \mathcal{N}^{\G}_\tau$ which is distinct from $\alpha$ and $\beta$. Then we see that $|V(\G)| = 4$, because no connected graph on three vertices has two disjoint nonloop edges, and if $|V(\G)| \geq 5$, then the connectedness of $\G$ implies that we would be able to find an edge $e_j$ which is not incident to both $e_\alpha, e_\beta$, and thus $j$ would be a witness to Equation \ref{AlmostAllIntersections}. So, $|V(\G)| = 4$, and $\G$ must have a full subgraph isomorphic to one of $\mathbb{E}_0, \ldots, \mathbb{E}_4$: the only remaining possibility for a full subgraph of a graph with four vertices is the star graph, but we see that any such graph does not have a pair of disjoint nonloop edges.

Now we consider the case where there exists a pair $\alpha, \beta \in \mathcal{N}^\G_\tau$ such that $\langle \alpha, \beta \rangle_\tau^\G = 1$ (i.e., $e_\alpha$ and $e_\beta$ are distinct nonloop edges of $\G$ which meet at a single vertex), and
\[\langle \alpha, \beta \rangle_\tau^\G < \langle \delta_j(\alpha), \delta_j(\beta) \rangle^{\G/e_j}_{\tau_j} \]
for any $j \in \mathcal{N}^{\G}_\tau$ which is distinct from $\alpha$ and $\beta$. Then we must either have that the above assumption is vacuous, that is, $\mathcal{N}^{\G}_\tau = \{\alpha, \beta \}$, so that $\G$ has a full subgraph isomorphic to $\mathbb{E}_6$, or that every nonloop edge of $\G$ is contained in a 3-cycle, so that $\G$ has a full subgraph isomorphic to $\mathbb{E}_5$.

We conclude the proof by considering the case where exactly one of $\alpha$ or $\beta$ is in $\mathcal{N}^\G_\tau$; without loss of generality, suppose $\alpha \in \mathcal{N}^\G_{\tau}$, while $\beta$ corresponds either to a loop or marking. Again suppose that 
\[\langle \alpha, \beta \rangle_\tau^\G < \langle \delta_j(\alpha), \delta_j(\beta) \rangle^{\G/e_j}_{\tau_j} \]
for any $j \in \mathcal{N}^{\G}_\tau$ which is distinct from $\alpha$ (such $j$ will exist because $|V(\G)| \geq 3$). Since $\beta$ corresponds to a marking or loop, we have $\langle \delta_j(\alpha), \delta_j(\beta) \rangle^{\G/e_j}_{\tau_j} \leq 1$ for any such $j$, so we have $\langle \alpha, \beta \rangle^\G_\tau = 0$. So in $\G$ we can find a nonloop edge $e_j$, and a loop $e_\ell$, respectively a marking $x$, such that $e_j$ and $e_\ell$, resp. $x$, are disjoint in $\G$, but their images meet whenever we contract a nonloop edge distinct from $e_j$. We deduce that $|V(\G)| = 3$, so $\G$ has a full subgraph isomorphic to one of $\mathbb{E}_5$ or $\mathbb{E}_6$; this completes the proof.
\end{proof}

\begin{figure}[h]
	\centering
	\includegraphics[scale=1]{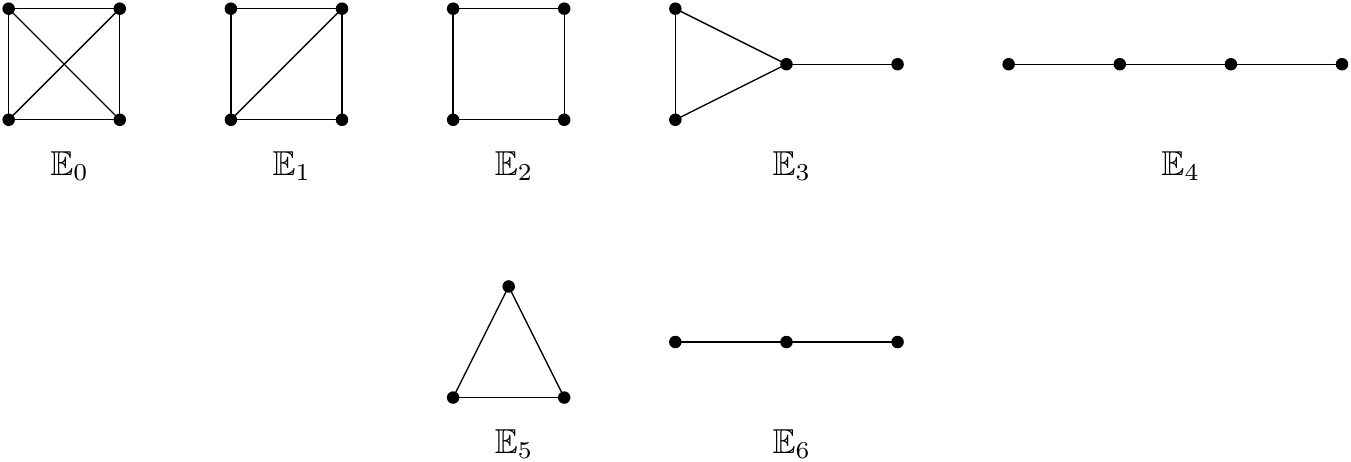}
	\caption{The exceptional graphs $\mathbb{E}_i$}
	\label{ExceptionalFamilies}
\end{figure}
We now establish the first half of Step 4: we can use $\D^\G_{\tau}$ and $\O^\G_\tau$ to decide whether $\G$ has a full subgraph isomorphic to any of the $\mathbb{E}_i$.
\begin{lem}\label{exceptionalcases}
	Suppose that $(\G, \tau)$ is a $\Gamma^{\EL}_{g, n}$-object with $b^1(\G) = g$. Then for each $i \in \{0, \ldots, 6\}$, there exists an explicit algorithm which takes as input $\D^\G_\tau$ and $\O^\G_\tau$, and checks whether $\G$ has a full subgraph isomorphic to $\mathbb{E}_i$.
\end{lem}	
\begin{proof}
	It turns out that the most difficult case is $i = 4$, so we save that for last. 
	\begin{enumerate}
		\item [$(\mathbb{E}_0)$] A graph $\G$ has  $\mathbb{E}_0$ as a full subgraph if and only if it has four vertices, a cycle of length four, and four distinct $3$-cycles $C_1, C_2, C_3, C_4 \subseteq E(\G)$ such that $\cup_{i = 1}^{4} C_i$ contains no $2$-cycles.
		\item [$(\mathbb{E}_1)$] A graph has $\mathbb{E}_1$ as a full subgraph if and only if it has four vertices, contains both a $3$-cycle and $4$-cycle, but does not belong to $\mathbb{E}_0$.
		\item [$(\mathbb{E}_2)$] A graph has $\mathbb{E}_1$ as a full subgraph if and only if it has four vertices, contains a $4$-cycle, and does not contain any $3$-cycles.
		\item [$(\mathbb{E}_3)$] A graph has $\mathbb{E}_3$ as a full subgraph if and only if it has four vertices, contains a $3$-cycle, and does not contain any $4$-cycles.
		\item [$(\mathbb{E}_5)$] A graph has $\mathbb{E}_5$ as a full subgraph if and only if it has three vertices and contains a $3$-cycle.
		\item [$(\mathbb{E}_6)$] A graph has $\mathbb{E}_6$ as a full subgraph if and only if it has three vertices and does not contain any $3$-cycles.
	\end{enumerate} 
To check if a graph $\G$ has $\mathbb{E}_4$ as a subgraph, we first check that it has four vertices and that it does not have any of $\mathbb{E}_0, \ldots, \mathbb{E}_3$ as a full subgraph, using the above rules. At this point we can be sure that $\G$ either has $\mathbb{E}_4$ or the star graph $\mathbb{T}_4$ on four vertices as a full subgraph. To differentiate between these two, we proceed in separate ways, depending on whether or not $\G$ contains either a loop or a marking.

If $\G$ contains a loop, indexed by $\ell \in [p]$, then since $|V(\G)| = 4$, we can compute
\[ \langle \ell, \alpha \rangle_\tau^\G \]
for all $\alpha \in [p] \sqcup I_n$, using the result of Lemma \ref{GenericCapping}, which implies that (\ref{AlmostAllIntersections}) correctly computes these numbers. If there exists three distinct elements $j_1, j_2, j_3 \in \mathcal{N}^\G_\tau$ such that no two of $j_1, j_2, j_3$ form a $2$-cycle, and such that $\langle \ell, j_s \rangle_\tau^\G = 1$ for $s = 1, 2, 3$, then $\G$ has $\mathbb{T}_4$ as a full subgraph. If we can find at most two such distinct elements $j_1, j_2 \in \mathcal{N}^\G_{\tau}$, then $\G$ has $\mathbb{E}_4$ as a full subgraph. If we can only find one such element, then $\G$ has $\mathbb{T}_4$ as a full subgraph if and only if there exists a nonloop index $j \in \mathcal{N}^\G_\tau$ such that $c_j(e_\ell)$ is contained in every nonloop edge of $\G/e_j$. Note that the same procedure works if $\ell$ is replaced by a marking.

If $\G$ contains no loops or markings and has a full subgraph isomorphic to either $\mathbb{T}_4$ or $\mathbb{E}_4$, then by stability, every single edge of $\G$ is contained in a $2$-cycle. Then $\G$ has a full subgraph isomorphic to $\mathbb{E}_4$ if and only if there exists a nonloop contraction $\G/e$ of $\G$, such that $\G/e$ has a loop which is not incident to all of its edges.

As such, we have shown that the data $\D^\G_\tau$ and $\O^\G_\tau$ indeed suffice to decide whether $\G$ has a full subgraph isomorphic to $\mathbb{E}_4$, and the lemma is proven.
\end{proof}	
	
The following lemma shows that Step 4 of our algorithm terminates; given that $\G$ has a full subgraph isomorphic to $\mathbb{E}_i$ for some $i = 0, \ldots, 6$, there exists an explicit procedure to calculate $\Q^\G_\tau$ from $\D^\G_\tau$ and $\O^\G_\tau$.
\begin{lem}\label{ExceptionalCalculations}
Suppose given that $b^1(\G) = g$, and that $\G$ has a full subgraph isomorphic to $\mathbb{E}_i$, for some $i = 0, \ldots, 6$. Then there is an explicit algorithm which takes as input $\D^\G_\tau$ and $\O^\G_\tau$ and calculates $\Q^{\G}_\tau$
\end{lem}
\begin{proof}
	When $|V(\G)| = 4$, the first two steps of the algorithm, taken together with Lemma \ref{GenericCapping}, allow us to calculate all of the numbers $\langle \alpha, \beta\rangle_\tau^\G$ except possibly when 
	\begin{quote}
	($\dagger$) $\alpha, \beta$ are distinct nonloop indices which do not correspond to a $2$-cycle in $\G$.
	\end{quote}
	Thus for $\mathbb{E}_0, \ldots, \mathbb{E}_4$, we need only give a rule for calculating these numbers. Throughout the first five parts of the proof, we  assume that $\alpha, \beta$ are indices satisfying ($\dagger$).\\
	
	\noindent\underline{If $\G$ has a full subgraph isomorphic to $\mathbb{E}_0$ or $\mathbb{E}_1$}: then \[\langle \alpha, \beta \rangle^\G_\tau = \begin{cases}
	1 &\text{ if } \text{there exists }j \in [p] \text{ such that }\{\alpha, \beta, j\} \text{ index a 3-cycle of } \G;\\
	0 &\text{ otherwise.}
	\end{cases}\] 
	
	\noindent\underline{If $\G$ has a full subgraph isomorphic to $\mathbb{E}_2$}: we proceed in two different ways, depending on whether $\G$ contains any loops or markings. 
	
	If it has either a loop or a marking index $\gamma \in [p] \sqcup I_n$, we first use Lemma \ref{GenericCapping} to calculate $\langle \gamma, \epsilon \rangle^\G_\tau$ for all $\epsilon \in [p] \sqcup I_n$. Because $\mathbb{E}_2$ is a $4$-cycle, we can find two distinct nonloop indices $j, k \in \mathcal{N}^\G_\tau$ such that $\{j, k\}$ does not index a $2$-cycle, and such that $\langle \gamma, j \rangle^\G_\tau = \langle\gamma, k \rangle^\G_\tau = 1$. In other words, $j$ and $k$ index two nonloop edges which are meeting at a corner of the square $\mathbb{E}_2$; this is the corner at which the loop/marking $\gamma$ is based. To calculate $\langle \alpha, \beta \rangle_\tau^\G$, then, it suffices to calculate $\langle \alpha, j \rangle^\G_\tau$, $\langle \alpha, k \rangle^\G_\tau$, $\langle \beta, j \rangle^\G_\tau$, and $\langle \beta, k \rangle^\G_\tau$. We can find $\langle \alpha, j \rangle_\tau^\G$ by looking at $\G/e_j$, and checking whether $c_j(e_{\alpha})$ meets the image of the loop or marking indexed by $\gamma$. If $c_j(e_\alpha)$ is a loop on the same vertex as the image of $\gamma$, then $\langle \alpha, j \rangle^\G_\tau = 2$. If the image of $e_\alpha$ is a nonloop edge which contains the image of $\gamma$, then $\langle \alpha, j \rangle^\G_\tau = 1$; in any other case $\langle \alpha, j \rangle^\G_\tau = 0$. By symmetry, the same rule can be adjusted to compute the other three numbers $\langle \alpha, k \rangle^\G_\tau$, $\langle \beta, j \rangle^\G_\tau$, and $\langle \beta, k \rangle^\G_\tau$.
	
	Now suppose that $\G$ has a full subgraph isomorphic to $\mathbb{E}_2$, but contains no loops or markings. If there exists $j \in \mathcal{N}^\G_\tau$ such that $\{\alpha, j\}$ forms a $2$-cycle, then $\langle \alpha, \beta \rangle^\G_{\tau}$ is equal to $1$ if $c_j(e_\alpha), c_j(e_\beta)$ meet in $\G/e_j$, and $0$ otherwise. Reversing the roles of $\alpha$ and $\beta$, we see that the same rule works if $\beta$ is contained in a $2$-cycle. If neither $\alpha$ nor $\beta$ is contained in a $2$-cycle, then $\langle \alpha, \beta \rangle_\tau^\G = 0$ by the stability condition: if they met each other at a vertex, that vertex would have weight $0$ and valence $2$ in $\G$, since $\G$ is supposed not to have loops or markings.\\
	
	\noindent\underline{If $\G$ has a full subgraph isomorphic to $\mathbb{E}_3$}: then $\langle\alpha, \beta \rangle^\G_\tau = 1$ if $e_\alpha, e_\beta$ are contained in a $3$-cycle of $\G$. At least one of $e_\alpha$ or $e_\beta$ must be contained in a $3$-cycle since we have assumed that $\{\alpha, \beta \}$ does not index a $2$-cycle. So suppose without loss of generality that $\alpha$ is contained in a $3$-cycle, and $\beta$ is not. Then by stability, there exists an index $\gamma \in [p] \sqcup I_n$ corresponding to one of the following: a parallel edge to $e_\beta$, a loop meeting $e_\beta$, or a marking supported on one of the vertices of $e_\beta$. Then we see that $\langle \alpha, \beta\rangle^\G_\tau = \langle \delta_{\beta}(\alpha), \delta_{\beta}(\gamma) \rangle^{\G/e_\beta}_{\tau_\beta}$.\\
	
	\noindent\underline{If $\G$ has a full subgraph isomorphic to $\mathbb{E}_4$}: we again use casework, depending on whether or not $\G$ has loops or markings. The argument is similar in spirit to the case of $\mathbb{E}_2$ and the last part of the proof of Lemma \ref{exceptionalcases}; the details are left to the reader.\\
	
	When $|V(\G)| = 3$, we have $\langle \alpha, \beta \rangle^\G_\tau = 1$ whenever $\alpha, \beta$ index distinct nonloop edges which do not form a $2$-cycle, so we need only show how to calculate $\langle \alpha, \beta \rangle^\G_\tau$ when
	\begin{quote}
		($*$) $\alpha \in \mathcal{N}^\G_\tau$ is the index of a nonloop edge, and $\beta \in ([p] \smallsetminus \mathcal{N}^\G_\tau) \sqcup I_n$ is the index of loop or marking.
	\end{quote}	
	In the remaining cases we assume $\alpha, \beta$ are indices satisfying ($*$).\\
	
	\noindent\underline{If $\G$ has a full subgraph isomorphic to $\mathbb{E}_5$}: we use the claim that if $\G$ is any object of $\Gamma_{g, n}$ with $b^1(\G) = g$, $e$ is a nonloop edge of $\G$, and $v$ is a vertex contained in $e$, we always have
	\[\val(v) + |m_\G^{-1}(v)| \lneq \val(c_e(e)) + |m_{\G/e}^{-1}(c_e(e))|, \]
	where $c_e(e) \in V(\G/e)$ is the image of $e$ in the contraction, and $m_\G, m_{\G/e}$ are the marking functions of $\G$ and $\G/e$, respectively. Indeed, due to the stability condition and our assumption that $b^1(\G) = g$, all vertices of $\G$ satisfy $\val(v) + |m_\G^{-1}(v)| \geq 3$. The claim then follows from the formula 
	\[\val(c_e(e)) + |m_{\G/e}^{-1}(c_e(e))| = \val(v) + \val(v') + |m_\G^{-1}(v)| + |m_\G^{-1}(v')| - 2, \]
	where $v'$ is the other vertex of $e$. 
	
	If $\G$ has a full subgraph isomorphic to $\mathbb{E}_5$, and $\beta$ indexes a marking or loop based at a vertex $v_\beta$, we find
	\[\val(v_\beta) + |m_{\G}^{-1}(v_\beta)| = \min_{j \in \mathcal{N}^\G_\tau} \{\val(v_{\delta_j(\beta)}) + |m_{\G/e_j}^{-1}(v_{\delta_j(\beta)})|  \} , \]
	where $v_{\delta_{j}(\beta)}$ denotes the vertex of $\G/e_j$ which supports the image of the loop or marking indexed by $\beta$. The above equality holds because if $\G$ contains a $3$-cycle, then there must be a nonloop edge $e_j$ which does not contain $v_\beta$. Now for $\alpha \in \mathcal{N}^\G_{\tau}$, we see that $\langle\alpha, \beta \rangle^\G_\tau = 1$ if and only if the valence plus number of markings supported by the vertex $v_{\delta_{\alpha}(\beta)}$ in $\G/e_\alpha$ is strictly greater than the valence plus markings supported of the vertex $v_\beta$ in $\G$. Otherwise, $\langle \alpha, \beta \rangle^\G_\tau = 0$.\\
	
	\noindent\underline{If $\G$ has a full subgraph isomorphic to $\mathbb{E}_6$}: First check if $\alpha$ is contained in a $2$-cycle $\{\alpha, j\}$. If so, then $\langle \alpha, \beta \rangle^\G_\tau = \langle \delta_{j}(\alpha), \delta_{j}(\beta) \rangle^{\G/e_{j}}_{\tau_{j}}$. Otherwise, if $\alpha$ is not contained in any $2$-cycle, we check if there is some other $2$-cycle $\{j, k\}$ of $\G$, and use the same rule we just described to check if $\langle \beta, j \rangle^\G_\tau = 1$; otherwise, if this quantity is $0$, we can conclude that $\langle \alpha, \beta \rangle^\G_\tau = 1$. 
	If $\langle \beta, j \rangle^\G_\tau = 1$, then there are two possibilities for a full edge-labelled subgraph of $\G$, as seen in Figure \ref{calculationE_6}.
	\begin{figure}[h]
		\centering
		\includegraphics[scale=1]{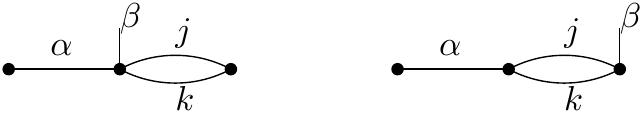}
		\caption{}
		\label{calculationE_6}
	\end{figure}	

	To decide between these two possibilities, we appeal to stability: $\langle \alpha, \beta \rangle^\G_\tau = 1$ if and only if there exists some $\gamma \in ([p] \smallsetminus \mathcal{N}^\G_\tau) \sqcup I_n$ such that $\langle \beta, \gamma \rangle^\G_\tau = 0$ (recall that this number is calculated by Lemma \ref{GenericCapping}), and such that $\langle \delta_{\alpha}(\beta), \delta_\alpha(\gamma) \rangle^{\G/e_\alpha}_{\tau_\alpha} = 1$.
	
	Finally, if $\G$ contains no $2$-cycles, then by stability each vertex of $\G$ must support a nonempty set of markings or loops, and again we have $\langle \alpha, \beta \rangle^\G_\tau = 1$ if and only if there exists $\gamma \in ([p] \smallsetminus \mathcal{N}^\G_\tau) \sqcup I_n$ such that $\langle \beta, \gamma \rangle^\G_\tau = 0$ and $\langle \delta_{\alpha}(\beta), \delta_\alpha(\gamma) \rangle^{\G/e_\alpha}_{\tau_\alpha} = 1$.
\end{proof}

Proposition \ref{DataTransfer} is now proven by running the algorithm depicted in Figure \ref{DecisionTree}.

\subsection{Proof of Proposition \ref{Reconstruction}} We conclude this section by proving Proposition \ref{Reconstruction}, which states that if $b^1(\G) = g$, then together $\D^\G_\tau$ and $\Q^\G_\tau$ determine $(\G, \tau)$ up to isomorphism. The gist is that we will recover $(\G, \tau)$ as an ``uncontraction" of some $(\G/e_j, \tau_j)$ in the list $\D^\G_\tau$ at a specified vertex. We now make precise what we mean by uncontraction. Recall that given an edge-labelled pair $(\G, \tau)$, we put $c_j : \G \to \G/e_j$ for the contraction of edge $e_j$ in $\Gamma_{g, n}$.
\begin{defn}
	Let $(\H, \pi)$ be an object of $\Gamma_{g, n}^\EL$, and fix a vertex $v \in V(\H)$. Then another $\Gamma^{\EL}_{g, n}$-object $(\G, \tau)$ is called a \textsl{\textbf{$j$-uncontraction}} of $(\H, \pi)$ at $v$ if there exists an isomorphism of pairs $\varphi: (\G/e_j, \tau_j) \to (\H, \pi)$ which satisfies $\varphi(c_j(e_j)) = v$.
\end{defn}
Given a vertex $v \in V(\H)$, we put $\mathcal{N}^v_\pi \subseteq [p]$ for the set of indices corresponding via $\pi: E(\H) \to [p]$ to nonloop edges containing $v$, and similarly put $\mathcal{L}^v_\pi \subseteq [p]$ for the set of indices corresponding to loops based at $v$. When $n \geq 1$, we also define $I^v_\pi \subseteq I_n$ to be the set of markings based at $v$.
\begin{lem}\label{HowtoUncontract}
Assume that $(\H, \pi)$ in $\Gamma^{\EL}_{g, n}$ satisfies $b^1(\H) = g$, and let $v \in V(\H)$. Then any $j$-uncontraction $(\G, \tau)$ of $(\H, \pi)$ at $v$ is determined up to isomorphism by three ordered partitions \[\mathcal{N}^{v}_\pi = \mathcal{N}^{v}_\pi(1) \sqcup \mathcal{N}^{v}_\pi(2),\] \[I^v_\pi = I^v_\pi(1) \sqcup I^v_\pi(2),\] and \[\mathcal{L}^v_\pi = \mathcal{L}^v_\pi(0) \sqcup \mathcal{L}^v_\pi(1) \sqcup \mathcal{L}^v_\pi(2),\] subject to the stability inequalities
\[|\mathcal{N}^v_\pi(i)| + 2|\mathcal{L}^v_\pi(i)| + |I^v_\pi(i)| + |\mathcal{L}^v_\pi(0)| + 1 \geq 3 \]
for $i = 1, 2$. Given a distinct set of three ordered partitions $\mathcal{N}^v_\pi = \mathcal{N}^v_\pi(1)' \sqcup \mathcal{N}^v_\pi(2)'$ etc., the resulting $j$-uncontractions $(\G, \tau)$ and $(\G', \tau')$ satisfy $\Q^\G_{\tau} = \Q^{\G'}_{\tau'}$ if and only if \[\mathcal{N}^v_\pi(1) = \mathcal{N}^v_\pi(2)'\] \[{I}^v_\pi(1) = I^v_\pi(2)',\]  \[\mathcal{L}^v_\pi(1) = \mathcal{L}^v_{\pi}(2)',\]
and 
\[\mathcal{L}^v_\pi(0) = \mathcal{L}^v_\pi(0)'. \]
Moreover, if the above equalities hold, then $(\G, \tau)$ and $(\G', \tau')$ are isomorphic.
\end{lem}
\begin{proof}
	Given such an ordered partition, we form $(\G, \tau)$ by removing the vertex $v$ from $\H$, and replacing it with an edge $e_j$ which will be labelled $j$ by $\tau$ (since $b^1(\H) = g$, any graph in $\Gamma_{g,n}$ which contracts to it must do so by a nonloop edge). Label the vertices of $e_j$ as $v_1$ and $v_2$. Then all the edges in $\H$ are relabelled by $\delta^j \circ \pi : E(\H) \to [p]$. For $i = 1, 2$, the edges, loops, and markings which now correspond to $\delta^j\mathcal{N}^v_\pi(i)$, $\delta^j\mathcal{L}^v_\pi(i)$, and $I^v_\pi(i)$, respectively, are shifted to vertex $v_i$ of $\G$. Finally, the loops in $\H$ now labelled as $\delta^j\mathcal{L}^v_\pi(0)$ are turned into nonloop edges, which become parallel to $e_j$ in $\G$. The stated inequalities are required to ensure that the resulting graph is stable. See Figure \ref{juncontract} to visualize this procedure.

\begin{figure}[h]
	\centering
	\includegraphics[scale=1]{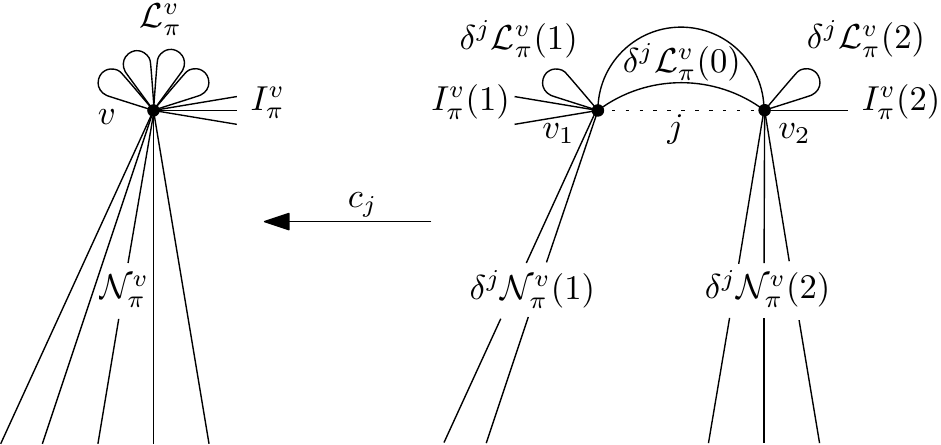}
	\caption{Constructing a $j$-uncontraction given three ordered partitions as in Lemma \ref{HowtoUncontract}}
	\label{juncontract}
\end{figure}	
	It is clear that if we switch the roles of $v_1$ and $v_2$ in the described process, we get an isomorphic edge-labelled graph. Conversely, suppose we are given two sets of ordered partitions as in the statement of the lemma, such that the resulting $j$-uncontractions satisfy $\Q^\G_\tau = \Q^{\G'}_{\tau'}$. Then it is immediate that we must have $\delta^j\mathcal{L}^v_\pi(0) = \delta^j\mathcal{L}^v_\pi(0)'$, for these will index the entries in the $j$th column of $\Q^\G_{\tau}$, resp. $\Q^{\G'}_{\tau'}$ which are equal to $2$. Next, we claim that $\mathcal{N}^v_\pi(1) = \mathcal{N}^v_\pi(1)'$ or $\mathcal{N}^v_\pi(1) = \mathcal{N}^v_\pi(2)'$. Otherwise, if $\mathcal{N}^v_\pi(1) \neq \mathcal{N}^v_\pi(1)'$, but these have an element $\alpha$ in common, then it would not be possible for the resulting intersection matrices to coincide: if, for example, there were $\beta \in \mathcal{N}^v_\pi(1)$ with $\beta \notin \mathcal{N}^v_{\pi}(1)'$, then we would necessarily have \[\langle \delta^j(\alpha), \delta^j(\beta) \rangle^{\G'}_{\tau'} \lneq \langle \delta^j(\alpha), \delta^j(\beta) \rangle^{\G}_{\tau}. \] We conclude the proof by noting that by the equality of intersection matrices, if $\mathcal{N}^{v}_\pi(1) = \mathcal{N}^v_\pi(2)'$, then we must also have $\mathcal{L}^{v}_\pi(1) = \mathcal{L}^v_\pi(2)'$ and $I^{v}_\pi(1) = I^v_\pi(2)'$, and similarly, if $\mathcal{N}^{v}_\pi(1) = \mathcal{N}^v_\pi(1)'$, then we also have $\mathcal{L}^{v}_\pi(1) = \mathcal{L}^v_\pi(1)'$ and $I^{v}_\pi(1) = I^v_\pi(1)'$.
\end{proof}
We now prove Proposition \ref{Reconstruction}, which states that any $(\G, \tau)$ in $\Gamma_{g, n}^{\EL}$ with $b^1(\G) = g$ is determined up to isomorphism by $\D^{\G}_\tau$ and $\Q^{\G}_\tau$.
\begin{proof}[Proof of Proposition \ref{Reconstruction}]
	Lemma \ref{HowtoUncontract} implies that if we are given $\Q^\G_{\tau}$, we can determine $(\G, \tau)$ up to isomorphism by finding a pair $(\H, \pi)$ and a vertex $v \in V(\H)$ such that $(\G, \tau)$ is a $j$-uncontraction of $(\H, \pi)$ at $v$. To do this, we again use the observation that since $b^1(\G) = g$, then for any nonloop edge $e \in E(\G)$ and vertex $v$ of $e$, we have
	\[ \val(v) + |m_{\G}^{-1}(v)| \lneq \val(c_e(e)) + |m_{\G/e}^{-1}(c_e(e))|. \]
	Therefore, we can take $(\H, \pi) = (\G/e_j, \tau_j)$, where $j$ is chosen so that $\G/e_j$ contains a vertex $\bar{v}$ of maximal valence plus number of markings, across all graphs in the list $\D^{\G}_\tau$. Then, by the above inequality and the maximality of $\bar{v}$, it must be that the valence of $\bar{v}$ is strictly greater than the valence of any vertex in $\G$. As such, $(\G, \tau)$ must be a $j$-uncontraction of $(\G/e_j, \tau_j)$ at $\bar{v}$, and we are done.
\end{proof}
Thus we have proven Theorem \ref{Reconstruct} in a constructive fashion: using the arguments in the proofs of Lemma \ref{HowtoUncontract} and Proposition \ref{Reconstruction}, together with the algorithmic proof (Figure \ref{DecisionTree}) of Proposition \ref{DataTransfer}, we get an explicit procedure to reconstruct $(\G, \tau)$ from $\mathcal{D}^{\G}_\tau$, given that $b^1(\G) = g$ and $|V(\G)| \geq 3$.
\section{Calculations in $\V^{2}_{g, n}$ and $\V^3_{g, n}$}\label{Calculations}
In this section we will prove Theorem \ref{Image}, which we restate here for convenience.
\begin{customthm}{1.6}
	Suppose that $g, n \geq 0$ and $2g - 2 + n \geq 3$, and let $\Phi \in \Aut(\Delta_{g, n})$. If $n = 0$, so $g \geq 2$, then we have
\[\Phi|_{\V^{2}_g} = \mathrm{Id}|_{\V^2_g}. \]
If $n \geq 1$, then there exists a unique element $\sigma \in S_n$ such that 
\[\Phi|_{\V^{2}_{g, n}} = f_\sigma|_{\V^2_{g, n}}. \]
\end{customthm}

In other words, Theorem \ref{Image} states that the image of the restriction map 
\[\rho^2_{g, n}: \Aut(\Delta_{g, n}) \to \Aut(\V^{2}_{g, n}) \]
is isomorphic to $S_n$. Together with the injectivity of the restriction map (Theorem \ref{TwoVertex}), Theorem \ref{Image} will imply the main result of the paper, that $\Aut(\Delta_{g,n}) \cong S_n$ for $2g - 2 + n \geq 3$.

As discussed earlier (e.g. in the proof of Theorem \ref{TwoVertex}), the purity of the subcomplex $\V^{i}_{g, n}$ implies that any automorphism thereof is completely determined by its action on the set of top-dimensional faces $\V^{i}_{g, n}[g + i - 2]$, which correspond to isomorphism classes of $\Gamma^{\EL}_{g,n}$-objects $(\G, \tau)$ satisfying $b^1(\G) = g$, $|V(\G)| = i$, and $|E(\G)| = g + i - 1$. 

Motivated by this observation, our first intermediate result towards the proof of Theorem \ref{Image} is that any $\Phi \in \Aut(\Delta_{g,n})$ preserves the ``weak" isomorphism class of a top-dimensional simplex $[\G, \tau] \in \V^{2}_{g, n}[g]$, in the sense that $(\G, \tau)$ and $(\Phi \G, \Phi \tau)$ are isomorphic up to changing their respective marking functions $m_\G: I_n \to V(\G)$, $m_{\Phi \G}: I_n \to V(\Phi \G)$. This motivates the following definition.
\begin{defn}\label{WeakIsomorphism}
	Given two objects $(\G, \tau), (\G', \tau')$ of $\Gamma_{g, n}^{\EL}$ with $\G = (G, w, m)$, and $\G' = (G', w', m')$, we say that $(\G, \tau)$ and $(\G', \tau')$ are \textsl{\textbf{weakly isomorphic}} if there exists a $\mathsf{Graph}_g$-isomorphism \[\varphi: (G, w) \to (G', w')\] making the diagram
	\[\begin{tikzcd}
	&E(G') \arrow[dr, "\tau'"] \arrow[rr, "\varphi^*"] & &E(G) \arrow[dl, "\tau"']\\
	& &\lbrack p \rbrack &
	\end{tikzcd} \]
	commute. Such a map $\varphi$ is called a \textsl{\textbf{weak isomorphism of pairs}}, and is denoted with a dashed arrow \[\varphi: (\G, \tau) \dashrightarrow (\G', \tau').\]
\end{defn}
When $n = 0$, two pairs are weakly isomorphic if and only if they are isomorphic. The point of this definition is that when $n \geq 1$, the notion of weak isomorphism of pairs helps us to characterize orbits of the $S_n$-action on $\Delta_{g, n}$. Indeed,
\begin{quote}($\star$)
	given $\Phi \in \Aut(\Delta_{g,n})$ and $[\G, \tau] \in \Delta_{g,n}[p]$, then there exists a permutation $\sigma$ such that
	\[[\Phi\G, \Phi \tau] = f_\sigma[\G, \tau] \]
	if and only if, upon any choice of equivalence class representatives, there exists a weak isomorphism of pairs $\varphi: (\G, \tau) \dashrightarrow (\Phi\G, \Phi\tau)$ such that
	\begin{equation}\label{MarkingsPreserved}
	|m_{\Phi\G}^{-1}(\varphi(v))| = |m_{\G}^{-1}(v)|
	\end{equation}
	for all $v \in V(\G)$.
\end{quote}
See Figure \ref{WeakExample} for an example of weakly isomorphic, but not isomorphic, pairs.
\begin{figure}[h]
	\centering
	\includegraphics[scale=1]{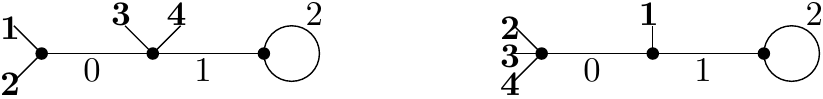}
	\caption{Two pairs in $\Gamma_{1, 4}^{\mathrm{EL}}$ which are weakly isomorphic but not isomorphic.}
	\label{WeakExample}
\end{figure}	

We now describe the strategy behind the proof of Theorem \ref{Image}. First, in \S \ref{WeakIsoPreserved} we will prove Proposition \ref{WeakPreserved}, which states that any $\Phi \in \Aut(\Delta_{g, n})$ preserves the weak isomorphism classes of simplices in $\V^{2}_{g, n}[g]$. This will prove Theorem \ref{Image} for $n = 0$. (Corollary \ref{Weak}).

In \S \ref{Orbits}, we make various counts of cells of $\V^{3}_{g, n}$ which have particular simplices $[\G, \tau]$ in $\V^{2}_{g, n}$ as faces; these counts will suffice to establish that the weak isomorphisms
\[\varphi: (\G, \tau) \dashrightarrow (\Phi \G, \Phi \tau) \]
for $[\G, \tau] \in \V^{2}_{g, n}[g]$ can be chosen to preserve the cardinality of markings on each vertex as in Equation \ref{MarkingsPreserved}. In other words, any $\Phi \in \Aut(\Delta_{g,n})$ must preserve the $S_n$-orbits of simplices in $\V^{2}_{g, n}[g]$ (Corollary \ref{SigmaExistence}). This also proves Theorem \ref{Image} for $n = 1$ (Corollary \ref{UniqueMarking}).

Finally, we treat the case $n \geq 2$ in \S\ref{PositiveGenus} and \S \ref{GenusZero}. To do so, we will make basic observations about the poset of geometric cells in $\Delta_{g, n}$, and study the action of $\Aut(\Delta_{g, n})$ on certain simplices of $\V^{3}_{g, n}$. After this we will show that any $\Phi \in \Aut(\Delta_{g,n})$ must act as a unique element of $S_n$ on $\V^{2}_{g, n}$ (Propositions \ref{FinalReformulation} and \ref{FinalZero}), thus completing the proof of Theorem \ref{Image}.

\subsection{Weak isomorphism classes of simplices are preserved (and the case $n = 0$)}\label{WeakIsoPreserved}
\begin{prop}\label{WeakPreserved}
	Suppose $\Phi \in \Aut(\Delta_{g,n})$, and let $[\G, \tau] \in \V^{2}_{g, n}[g]$. Then, for any representatives $(\G, \tau), (\Phi\G, \Phi\tau)$, there exists a weak isomorphism
	\[\varphi: (\G, \tau) \dashrightarrow (\Phi\G, \Phi\tau). \]
\end{prop}
\begin{proof}
	Let $[\G, \tau] \in \V^{2}_{g, n}[g]$. Since $b^1(\G) = g$ and $\Phi$ preserves loop indices (Corollary \ref{CyclesPreserved}), it suffices to show that $i, j \in [g]$ correspond to loops on the same vertex of $\G$ if and only if they correspond to loops on the same vertex of $\Phi \G$. To do this, we use the fact that $\Phi$ preserves loops together with $\sS$-equivariance. Indeed, if $i, j$ label loops on the same vertex of $\G$, then both must label loops of $\Phi\G$, but we also must have $(i, j) \in \Stab_{\sS_{g + 1}}[\Phi\G, \Phi\tau]$, i.e. there must be an automorphism of $\Phi\G$ transposing the two loops labelled by $i$ and $j$, but which fixes every other edge of $\Phi \G$.
	
	Suppose toward a contradiction that $i$ and $j$ label loops on distinct vertices of $\Phi \G$. Then the only way for $\Phi\G$ to have an automorphism of the required type is if $n = 0$, and if $\Phi \G$ has a full subgraph as in Figure \ref{WeakIso1}, where $i$ and $j$ are the only loop indices of $\Phi\G$. Note that $\Phi \G$ must have at least three nonloop edges by stability, because in this situation, $\G$ must also have $i$ and $j$ as its only loop indices, but they lie on a single vertex. 
	
	\begin{figure}[h]
	\centering
	\includegraphics[scale=1]{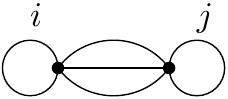}
	\caption{}
	\label{WeakIso1}
	\end{figure}
	
	Therefore, there exists a graph $\hat{\G}$, which has a full subgraph isomorphic to the one in Figure \ref{nZeroKey}, and no additional loops, which contracts to $\G$, and such that $\Phi \hat{\G}$ contracts to $\Phi \G$. 
	
	\begin{figure}[h]
		\centering
		\includegraphics[scale=1]{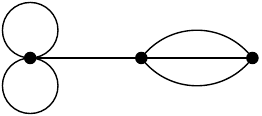}
		\caption{}
		\label{nZeroKey}
	\end{figure}
	
	Since it contracts to $\Phi \G$, the graph $\Phi \hat{\G}$ cannot have two loops on the same vertex. Because $|V(\Phi\hat{\G})| \geq 3$, this implies $\Phi\hat{\G}$ cannot have any automorphisms which transpose two loops while leaving every other edge fixed. However, $\hat{\G}$ clearly has this property. Since $\Phi$ preserves loops, this contradicts $\sS$-equivariance: for any edge-labelling $\hat{\tau}$ of $\hat{\G}$, there will be an element in the $\sS_{g + 2}$-stabilizer of $[\hat{\G}, \hat{\tau}]$ which is not in the $\sS_{g + 2}$-stabilizer of $[\Phi\hat{\G}, \Phi\hat{\tau}]$.
\end{proof}

Since the notions of weak isomorphism and $\Gamma_{g, n}^{\EL}$-isomorphism coincide when $n = 0$, an immediate corollary of Proposition \ref{WeakPreserved} is the $n = 0$ case of Theorem \ref{Image}, and hence Theorem \ref{Main}.

\begin{cor}\label{Weak}
	Fix $g \geq 2$. Then for any $\Phi \in \Aut(\Delta_{g})$, we have
	\[\Phi|_{\V^{2}_g} = \mathrm{Id}|_{\V^2_{g}}, \]
	hence $\Aut(\Delta_{g})$ is trivial.
\end{cor}

\subsection{The case $n \geq 1$: $\Aut(\Delta_{g,n})$ preserves $S_n$-orbits}\label{Orbits}
As we have now proven Theorem \ref{Image}, and hence Theorem \ref{Main}, when $n = 0$, we assume $n \geq 1$ for the remainder of the paper. 

To prove Theorem \ref{Image} in this case, it is convenient to introduce notation parameterizing the genus $g$ graphs in $\Gamma_{g, n}$ with two vertices. Fix a vertex set $\{v_1, v_2 \}$. For two integers $k, \ell$ such that $k, \ell\geq 0$ and $k + \ell \leq g$, we fix $B^{k, \ell}$ to be the graph with vertex set $\{v_1, v_2\}$, where $v_1$ and $v_2$ are connected by $g - (k + \ell) + 1$ edges, and such that $v_1$ supports $k$ loops while $v_2$ supports $\ell$ loops. By construction, $B^{k, \ell}$ has genus $g$ and $g + 1$ edges, and we have graph isomorphisms $B^{k, \ell} \cong B^{\ell, k}$. Up to isomorphism, any graph in $\Gamma_{g, n}$ with two vertices and $g + 1$ edges is a marked version of $B^{k, \ell}$ for some $k, \ell$. We say a subset $A \subseteq I_n$ is \textsl{\textbf{$(k, \ell)$-stable for $(g, n)$}} if, when we endow the graph $B^{k, \ell}$ with the marking function \[m^{k, \ell}_A: I_n \to V(B^{k, \ell})\] defined by
\begin{equation}\label{Marking}
m^{k, \ell}_A(x) = \begin{cases} v_1 &\text{ if }x\in A\\
v_2 &\text{ otherwise,}
\end{cases}
\end{equation}
 the resulting marked graph defines an object of $\Gamma_{g, n}$. 
\begin{defn}
	A triple $(k, \ell, A)$ is \textsl{\textbf{$(g, n)$-admissible}} if $k, \ell \geq 0$, $k + \ell \leq g$, and $A$ is $(k, \ell)$-stable for $(g, n)$.
\end{defn}
When we have implicitly or explicitly fixed $(g, n)$, we will shorten the phrase ``$(k, \ell)$-stable for $(g, n)$" into ``$(k, \ell)$-stable" and write ``admissible" rather than ``$(g, n)$-admissible." By our definition, $A$ is $(k, \ell)$-stable if and only if $A^c$ is $(\ell, k)$-stable. Given a $(g, n)$-admissible triple $(k, \ell, A)$, we define a $\Gamma_{g,n}$-object
\[\B_{A}^{k, \ell} = (B^{k, \ell}, w, m^{k, \ell}_A)\]
where $w: V(B^{k, \ell}) \to \Z_{\geq 0}$ is identically $0$, and $m^{k, \ell}_A : I_n \to V(B^{k, \ell})$ is the marking function defined in (\ref{Marking}).

Any pair $(\G, \tau)$ in $\Gamma_{g, n}^{\EL}$ with $|V(\G)| = 2$ and $b^1(\G) = g$ is isomorphic to a pair of the form $(\B_{A}^{k, \ell}, \pi)$ for some edge-labelling $\pi$ of $\B_{A}^{k, \ell}$. Moreover, given any $\Phi \in \Aut(\Delta_{g, n})$, Proposition \ref{WeakPreserved} implies that we can choose a $(k, \ell)$-stable subset $\Phi(A)$ such that an equality
\begin{equation}\label{Phi(A)}
\Phi[\B^{k, \ell}_A, \pi] = [\B^{k, \ell}_{\Phi(A)}, \pi]
\end{equation}
holds; here we are using that the edge sets of both graphs are identified with $E(B^{k, \ell})$. When $(k, \ell) = (0, 0)$, the choice of $\Phi(A)$ in (\ref{Phi(A)}) is not unique because we have $(\B^{0,0}_A, \pi) \cong (\B^{0,0}_{A^c}, \pi)$ for any edge-labelling of $\B^{0,0}_A$. We emphasize that when $(k, \ell) \neq (0, 0)$, the choice of $\Phi(A)$ is unique: the pairs $(\B^{k, k}_{A}, \pi)$ and $(\B^{k, k}_{A^c}, \pi)$ are not isomorphic for $k > 0$ because of the way we defined the marking functions in (\ref{Marking}).

To prove that any $\Phi \in \Aut(\Delta_{g,n})$ preserves $S_n$-orbits, we need to show that for all $(k, \ell)$-stable $A \subseteq I_n$, the choice of $\Phi(A) \subseteq I_n$ in (\ref{Phi(A)}) can be made so that $|A| = |\Phi(A)|$.	We do this by establishing some numerical invariants of simplices in $\V^{2}_{g, n}$ that must be preserved by automorphisms, and which uniquely determine $S_n$-orbits. When $g = 0$, this approach is the same as that used by Abreu and Pacini ~\cite{AbreuPacini} to prove that $\Aut(\Delta_{0,n})$ preserves $S_n$-orbits. The idea of counting uncontractions of graphs also arises in Bridson and Vogtmann's work on symmetries of Outer space ~\cite{bridson2001}.
\begin{defn}
	For $g \geq 0, n \geq 1$ such that $2g - 2 + n \geq  3$, we define a function 
	\[\mu_{g, n}: \{(g, n)\text{-admissible triples }(k, \ell, A) \text{ with } k \geq \ell\} \to \Z_{\geq 0} \]
	as follows:
	\begin{itemize}
		\item when $\ell = 0$, we let $\mu_{g, n}(k, \ell, A)$ be the number of distinct isomorphism classes of graphs $\G$ in $\Gamma_{g, n}$ such that $|V(\G)| = 3$, $\G$ has neither $3$-cycles nor loops, and there exists an edge $e$ of $\G$ such that $\G/e \cong \B^{k, \ell}_A$.
		\item when $k, \ell > 0$, we let $\mu(k, \ell, A)$ be the number of distinct isomorphism classes of graphs $\G$ in $\Gamma_{g, n}$ such that $|V(\G)| = 3$, $\G$ has neither $3$-cycles nor loops, and there exists an edge $e$ of $\G$ such that \[\G/e \cong \B^{k, \ell}_A/\{\text{all }\ell \text{ loops at }v_2 \}\]
	\end{itemize}
\end{defn}
The basis for the above definition of $\mu_{g, n}$ is that it is preserved under automorphisms of $\Delta_{g, n}$.
\begin{prop}\label{MuPreservation}
	Suppose $g \geq 0, n \geq 1$ such that $2g - 2 + n \geq 3$, and fix $\Phi \in \Aut(\Delta_{g, n})$. Let $(k, \ell, A)$ be an admissible triple with $k \geq \ell$, and let $\pi: E(B^{k, \ell}) \to [g]$ be an arbitrary edge-labelling. Find $\Phi(A) \subseteq I_n$ such that
	\[\Phi[\B^{k, \ell}_A, \pi] = [\B^{k, \ell}_{\Phi(A)}, \pi].  \]
	Then we have
	\[\mu_{g, n}(k, \ell, A) = \mu_{g, n}(k, \ell, \Phi(A)).  \]
\end{prop}
\begin{proof}
	In the first case, when $\ell = 0$, we have that $\mu_{g, n}(k, \ell, A)$ is the number of $\sS_{g + 2}$-orbits in $\V^{3}_{g, n}[g + 1]$ which contain an element $[\G, \tau]$ such that $\G$ contains no $3$-cycles or loops, and such that $[\G, \tau]$ has $[\B^{k, \ell}_A, \pi]$ as a face. 
	
	In the second case, when $k, \ell > 0$, set $L \subseteq [p]$ for indices corresponding via $\pi$ to loops at $v_2$. Then $\mu_{g, n}(k, \ell, A)$ is the number of $\sS_{g  - \ell + 2}$-orbits in $\V^{3}_{g, n}[g - \ell + 1]$ which contain an element $[\G, \tau]$ such that $\G$ contains no $3$-cycles or loops, and such that $[\G, \tau]$ has $d_L[\B^{k, \ell}_A, \pi]$ as a face. 
	
	After reinterpreting $\mu_{g, n}$ in these ways, the proposition follows from $\sS$-equivariance, together with Corollary \ref{CyclesPreserved}, which states that any $\Phi \in \Aut(\Delta_{g,n})$ preserves the cycles of a pair $(\G, \tau)$.
\end{proof}
We now compute $\mu_{g, n}(k, \ell, A)$ for all admissible triples $(k, \ell, A)$.
\begin{prop}\label{Formulae}
	The following formulas hold for admissible triples $(k, \ell, A)$ with $k \geq \ell$:
	\begin{enumerate}
		\item when $(k, \ell) = (0, 0)$ and $g = 0$, we have
		\[\mu_{g, n}(0,0, A) = 2^{|A|} + 2^{n - |A|} - 2n - 4; \]
		\item when $(k, \ell) = (0, 0)$ and $g \geq 1$, we have 
		\[\mu_{g, n}(0,0, A) = 2^{|A|} + 2^{n - |A|} - n - 2; \]
		\item if $k = 1$, then
		\[\mu_{g, n}(k, \ell, A) = 2^{|A|} - 1; \]
		\item if $k \geq 2$, then
		\[\mu_{g, n}(k, \ell, A) = 2^{|A|}. \]
	\end{enumerate}
\end{prop}
\begin{proof}
Let $\mathcal{P}$ be the property of having no $3$-cycles or loops. The formula for $\mu_{g, n}(k, \ell, A)$ is computed by describing maximal sets of pairwise nonisomorphic graphs having property $\mathcal{P}$ and contracting to $\H$, where $\H$ is a choice of graph depending on the triple $(k, \ell, A)$. When $\ell = 0$, we have
	\[\H = \B^{k, 0}_A, \]
	and when $\ell > 0$, we have
	\[\H = \B^{k, \ell}_A /\{\text{all }\ell \text{ loops at }v_2 \}. \]
	It is straightforward to classify isomorphism classes of stable graphs $\G$ having $\mathcal{P}$ and contracting to $\H$; see Figure \ref{mu-g-n}. The formulas of $\mu_{g, n}$ are then calculated by counting partitioins $I_n = A_1 \cup A_2 \cup A_3$ subject to the constraints in column 4 of Figure \ref{mu-g-n}. For example, we see that if $k \geq 2$ and $A$ is $(k, \ell)$-stable, then $\mu_{g, n}(k, \ell, A)$ is given by the number of partitions $I_n = A_1 \cup A_2 \cup A_3$ such that $A_2 \cup A_3 = A$; this is of course equal to the number $2^{|A|}$ of subsets of $A$. The other formulas are computed similarly by accounting for the additional constraints which appear in the fourth column when $k \leq 1$.
\end{proof}	
	\begin{figure}[h]
		\centering
		\includegraphics[scale=1]{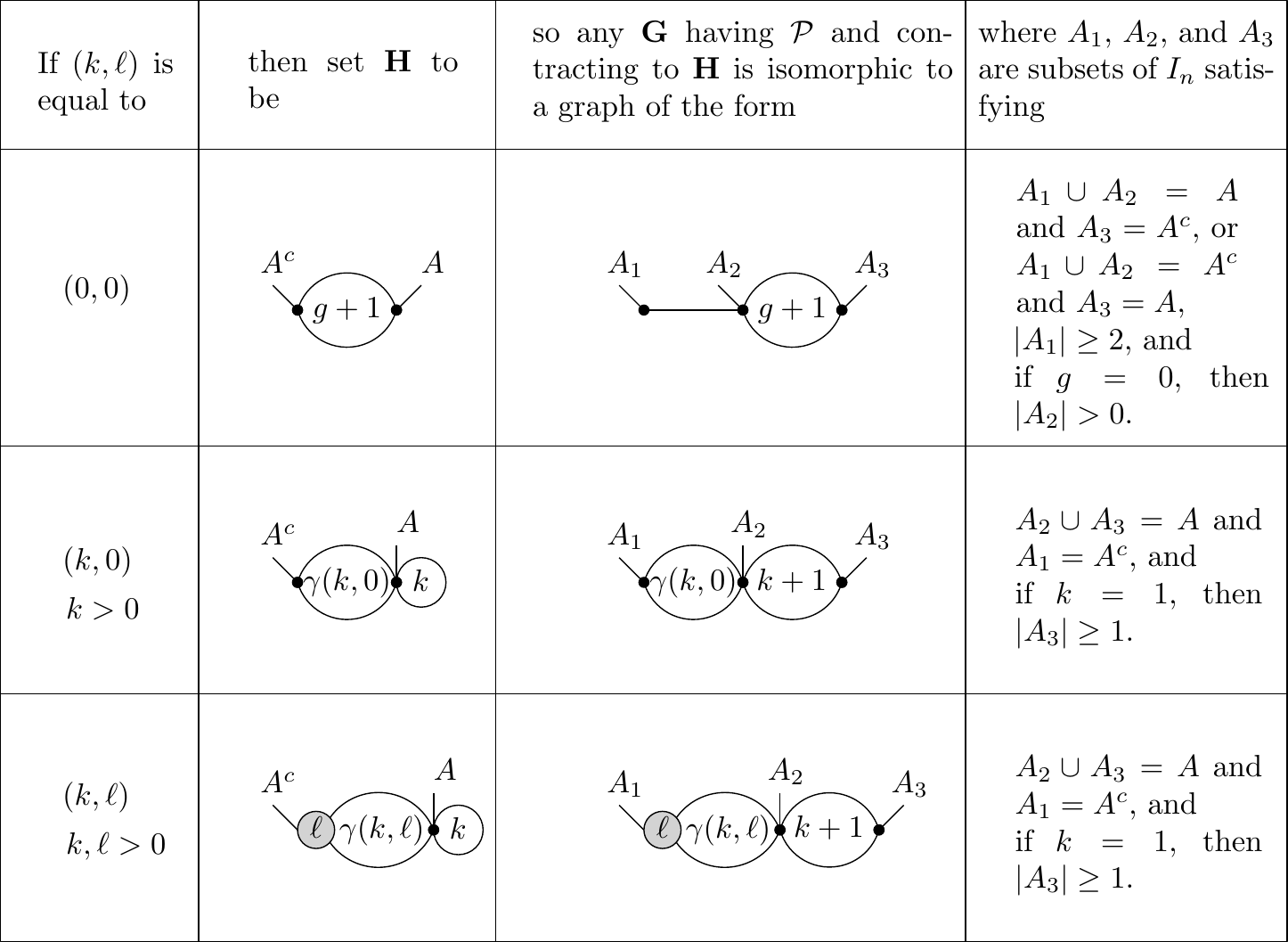}
		\caption{Computing $\mu_{g, n}$ for fixed $(k, \ell, A)$ as in Proposition \ref{Formulae}. In each case, $\mu_{g, n}(k, \ell, A)$ is given by the number of partitions $I_n = A_1 \cup A_2 \cup A_3$ satisfying the conditions in the fourth column. Here we have set $\gamma(k, \ell) : = g - k - \ell + 1$, and adopted the convention that a loop containing the number $k$ means that there are $k$ loops based at that vertex. $\mathcal{P}$ is the property of having neither $3$-cycles nor loops.}
		\label{mu-g-n}
	\end{figure}

The appeal of the formulas for $\mu_{g, n}$ given by Proposition \ref{Formulae} is that unless $(k, \ell) = (0, 0)$, they are uniquely determined by $|A|$. This observation allows us to prove that $\Aut(\Delta_{g,n})$ preserves $S_n$-orbits of simplices in $\V^{2}_{g, n}[g]$.
\begin{cor}\label{SigmaExistence}
	Suppose $g \geq 0, n\geq 1$, and $2g - 2 +n \geq 3$. Fix $\Phi \in \Aut(\Delta_{g, n})$. Then, given any $(g, n)$-admissible triple $(k, \ell, A)$, we can find $\Phi(A) \subseteq I_n$ such that $|A| = |\Phi(A)|$ and
	\[\Phi[\B^{k, \ell}_A, \pi] = [\B^{k, \ell}_{\Phi(A)}, \pi]. \]
	This choice of $\Phi(A)$ is unique unless $(k, \ell) = (0, 0)$ and $|A| = n/2$.
\end{cor}
\begin{proof}
	Replacing $A$ by $A^c$ if necessary, we may suppose that $k \geq \ell$ (recall that $(k, \ell, A)$ is admissible if and only if $(\ell, k, A^c)$ is admissible). Let $\Phi(A)$ be any choice of subset such that 
	\[\Phi[\B^{k, \ell}_A, \pi] = [\B^{k, \ell}_{\Phi(A)}, \pi]. \]
	Then Proposition \ref{MuPreservation} gives that
	\[\mu_{g, n}(k, \ell, A) = \mu_{g, n}(k, \ell, \Phi(A)). \]
	If $k \geq 1$, it follows from Proposition \ref{Formulae} that
	\[2^{|A|} = 2^{|\Phi(A)|}, \]
	which implies that $|A| = |\Phi(A)|$ and proves the existence part of the claim. If $(k, \ell) = (0, 0)$, we get that
	\[2^{|A|} + 2^{n - |A|} = 2^{|\Phi(A)|} + 2^{n - |\Phi(A)|}, \]
	so either $|A| = |\Phi(A)|$ or $|A^c| = |\Phi(A)|$. Since $[\B_{\Phi(A)}^{0,0}, \pi]= [\B_{\Phi(A)^c}^{0,0}, \pi]$, the existence is also proven in this case.
	
	For uniqueness, we have already discussed why $\Phi(A)$ is unique when $(k, \ell) \neq (0, 0)$. If $(k,\ell) = (0,0)$, then the duplication $[\B_{\Phi(A)}^{0,0}, \pi]= [\B_{\Phi(A)^c}^{0,0}, \pi]$ is eliminated upon requiring that $|\Phi(A)| = |A|$, unless it so happens that $|A| = n/2$.
\end{proof}	
As an immediate corollary, we have the $n = 1$ case of Theorem \ref{Main}.
\begin{cor}\label{UniqueMarking}
Suppose $g \geq 2$, and let $\Phi \in \Aut(\Delta_{g, 1})$. Then
\[\Phi|_{\V^{2}_{g,1}} = \mathrm{Id}|_{\V^{2}_{g, 1}}. \]
In particular, \[\Aut(\Delta_{g,1}) = \{\mathrm{Id} \} \cong S_1. \]
\end{cor}
\subsection{Finishing the proof when $g \geq 1$ and $n \geq 2$}\label{PositiveGenus}
All cases where $2g - 2 + n \geq 3$ and $n \leq 1$ of Theorem \ref{Image} have been proven. When $g \geq 1$ and $n \geq 3$, we have $n$ distinct stable graphs $\B^{g, 0}_{x}$ for $x \in I_n$ (we use $x$ instead of $\{x\}$ for ease of notation). Fixing some edge-labelling $\pi:E(B^{g, 0}) \to [g]$, we extract a unique permutation $\sigma \in S_n$ from $\Phi \in \Aut(\Delta_{g,n})$ by looking at the action of $\Phi$ on the $n$ simplices $[\B^{g, 0}_{x}, \pi]$; these form a closed orbit of $\Aut(\Delta_{g,n})$ by Corollary \ref{SigmaExistence}. That is, when $g \geq 1$ and $n \geq 3$, there is a unique $\sigma \in S_n$ such that
\[\Phi[\B^{g, 0}_{x}, \pi] = f_{\sigma}[\B^{g, 0}_{x}, \pi]\] for all $x \in I_n$. In this situation, we can finish the proof of Theorem \ref{Image} by proving that the composition $f_{\sigma^{-1}} \circ \Phi$ acts as the identity on $\V^{2}_{g, n}[g]$. When $n = 2$ and $g \geq 2$ (recall that the case $(g, n) = (1, 2)$ was treated in Example \ref{SmallExample}), we adjust this idea by replacing $\B^{g, 0}_{x}$ with $\B^{g - 1, 1}_{x}$. The upshot of this discussion is that the remaining cases of Theorem \ref{TwoVertex} (and hence of our main result Theorem \ref{Main}) when $g \geq 1$ are equivalent to the following proposition.
\begin{prop}\label{FinalReformulation}
Suppose $g \geq 1$, $n \geq 2$, and fix $\Phi \in \Aut(\Delta_{g,n})$.
\begin{enumerate}[(a)]
	\item if $n \geq 3$, $\pi : E(B^{g, 0}) \to [g]$ is any edge-labelling, and $\Phi$ fixes the $n$ simplices $[\B^{g, 0}_x, \pi]$ for $x \in I_n$, then \[\Phi|_{\V^{2}_{g, n}} = \mathrm{Id}|_{\V^{2}_{g, n}}\] 
	\item if $g \geq 2$, $n = 2$, and $\pi : E(B^{g - 1, 1}) \to [g]$ is any edge-labelling, and $\Phi$ fixes $[\B^{g-1, 1}_{x}, \pi]$ for $x = 1, 2$, then
	\[\Phi|_{\V^{2}_{g, 2}} = \mathrm{Id}|_{\V^{2}_{g, 2}}. \]
\end{enumerate}	
\end{prop}

The case $g = 0$ requires a slightly different strategy, as the only $(0, n)$-admissible triples $(k, \ell, A)$ are those where $k = \ell = 0$ and $|A|, |A^c| \geq 2$. We will first prove Proposition \ref{FinalReformulation} and then conclude the paper with the $g = 0$ case. To prove Proposition \ref{FinalReformulation}, we need some preliminary combinatorial observations. Given $A \subseteq I_n$, we put
\[\binom{A}{2} = \{S \subseteq I_n \mid |S| = 2 \text{ and }S\subseteq A\} \]
as in Definition \ref{AuxiliaryInvariants}, and define
\[\binom{A|A^c}{2} \defeq \binom{A}{2} \cup \binom{A^c}{2}. \]
Observe that when $|A| \geq 2$, the set $\binom{A}{2}$ completely determines $A$, while $\binom{A|A^c}{2}$ determines $A$ up to complementation: if
\[\binom{A|A^c}{2} = \binom{B|B^c}{2}, \]
then either $A = B$ or $A = B^c$.

Introduce a relation $\leftrightarrow$ on graphs in $\Gamma_{g, n}$ by declaring $\G \leftrightarrow \H$ if and only if there exists a graph $\G'$ in $\Gamma_{g, n}$ such that there are morphisms $\G' \to \G$ and $\G' \to \H$. We upgrade this to a relation on simplices of $\Delta_{g, n}$ by declaring $[\G, \tau] \leftrightarrow [\H, \pi]$ if and only if $\G \leftrightarrow \H$. By $\sS$-equivariance, it is clear that given $\Phi \in \Aut(\Delta_{g, n})$, we have $[\G, \tau] \leftrightarrow [\H, \pi]$ if and only if $\Phi[\G, \tau] \leftrightarrow \Phi[\H, \pi]$.

Using the definition of the geometric realization functor (\ref{Realization}) and the correspondence between $\sS_{p + 1}$-orbits in $\Delta_{g, n}[p]$ and isomorphism classes of graphs in $\Gamma_{g, n}$ with $p + 1$ edges, we can view $\leftrightarrow$ as a relation on cells of the geometric realization of $\Delta_{g, n}$: two distinct cells of the same dimension are related if and only if they are common faces of a larger one.

Given $\G$ in $\Gamma_{g, n}$, we say $e \in E(\G)$ is a \textsl{\textbf{$(g, 0)$-bridge}}, if upon contracting all edges of $E(\G)$ besides for $e$ and forgetting all markings, the resulting graph has one edge and two vertices, where one vertex is of weight $g$ and the other is of weight $0$.

\begin{lem}\label{Relationg0}
	Fix $g \geq 1$. Suppose $A \subseteq I_n$ is $(g, 0)$-admissible with $|A| \geq 2$. Then we have $\B^{g, 0}_x \leftrightarrow \B^{g, 0}_A$ if and only if $x \in A$. Thus if $\Phi \in \Aut(\Delta_{g,n})$ fixes $[\B^{g, 0}_x, \pi]$ for all $x \in I_n$, then $\Phi$ also fixes $[\B^{g, 0}_A, \pi]$ for any $(g, 0)$-admissible $A \subseteq I_n$.
\end{lem}
\begin{proof}
	It is straightforward to argue that any graph with three vertices which contracts to both $\B^{g, 0}_{x}$ and $\B^{g, 0}_A$ must have two distinct $(g, 0)$-bridges, and $g$ loops supported on one of the end vertices. If $A_1$ denotes the set of markings supported on the vertex with loops and $A_2$ denotes the set of markings supported on the middle vertex, we see that we must have $A_1 \cup A_2 = A$ and $A_1 = \{x\}$, from which the lemma follows.
\end{proof}
\begin{lem}\label{Relation}
	Fix $g \geq 1$, Suppose $A$ is $(g, 0)$-admissible and $B$ is $(k, \ell)$-admissible for some $k, \ell \geq 0$ with $k + \ell \leq g$ and $(k, \ell) \neq (g, 0)$. Then $\B^{g, 0}_A \leftrightarrow \B^{k, \ell}_B$ if and only $|A^c| \geq 2$ and $A^c \subseteq B$ or $A^c \subseteq B^c$.
\end{lem}	
\begin{proof}
	A graph with three vertices contracting to both $\B_A^{g, 0}$ and $\B_B^{k, \ell}$ must have a unique $(g,0)$-bridge which contains a leaf vertex; this leaf must support the elements of $A^c$. This observation suffices to prove the claim.
\end{proof}
Before the proof of \ref{FinalReformulation}, we need one final lemma to control the $\Aut(\Delta_{g,n})$-action on certain simplices of $\V^{3}_{g, n}[g + 1]$.
\begin{lem}\label{Snails}
	Suppose $g \geq 2$, and $[\G, \tau] \in \V^{3}_{g, n}[g + 1]$ is a simplex such that $\G$ has no $3$-cycles, and no $(g, 0)$-bridges. Then there exists a weak isomorphism of pairs
	\[\varphi: (\G, \tau) \dashrightarrow (\Phi \G, \Phi \tau) \]
\end{lem}	
\begin{proof}
	Suppose that $[\G, \tau]$ is as given in the proposition, and for all $i \in [g + 1]$ set \[\Phi e_i \defeq (\Phi \tau)^{-1}(i) \in E(\Phi \G).\] Then $\Phi \G$ also has no $3$-cycles. Moreover, $\Phi \G$ cannot have any $(g, 0)$-bridges: if for example edge $j$ were a $(g, 0)$-bridge in $\Phi \G$, then $\delta_i(j)$ would be a $(g, 0)$-bridge of $\Phi \G / \Phi e_i$ for some $i \neq j$ indexing a nonloop edge of both graphs. But $(\G/e_i, \tau_i)$ and $(\Phi \G / \Phi e_i, (\Phi\tau)_i)$ must be weakly isomorphic by Proposition \ref{WeakPreserved} and $\G/e_i$ cannot have a $(g, 0)$-bridge, since one cannot introduce such a bridge by contracting.
	
	To construct a weak isomorphism $(\G, \tau) \dashrightarrow (\Phi \G, \Phi \tau)$ choose a pair of nonloop indices $i, j$ of $\G$ such that $e_i$ and $e_j$ do not form a $2$-cycle; label the vertex which they share as $v_2$, the other endpoint of $e_i$ as $v_1$, and the other endpoint of $e_j$ as $v_3$. Then, since $\Phi e_i$ and $\Phi e_j$ cannot form a $2$-cycle in $\Phi \G$, we may deduce that $\G$ and $\Phi \G$ have full subgraphs as in Figure \ref{WeakIso3}.
	
	\begin{figure}[h]
		\centering
		\includegraphics[scale=1]{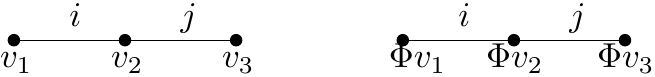}
		\caption{}
		\label{WeakIso3}
	\end{figure}	
	
Any index $\alpha \in [p]$ with $\alpha \neq i, j$ indexes an edge from $v_1$ to $v_2$, respectively $v_2$ to $v_3$, in $\G$ if and only if it indexes an edge from $\Phi v_1$ to $\Phi v_2$, resp. from $\Phi v_2$ to $\Phi v_3$, in $\Phi \G$, because $\Phi$ preserves $2$-cycles. To finish the proof of the lemma, it is enough to show that for each $s = 1, 2, 3$, an index $\ell \in [p]$ indexes a loop at vertex $v_s \in V(\G)$ if and only if it indexes a loop at vertex $\Phi v_s \in V(\Phi \G)$. By symmetry, it suffices to prove the claim for $s = 1$: the same proof will work for $s = 3$, and then it is forced to hold for $s = 2$.
	
	Indeed, suppose for sake of contradiction that we have some $\ell \in [p]$ which indexes a loop based at $v_1$ in $\G$, and also indexes a loop based at either $\Phi v_2$ or $\Phi v_3$ in $\Phi \G$. By Proposition \ref{Weak} and Corollary \ref{SigmaExistence}, there is a weak isomorphism $\varphi_j: (\G/e_j, \tau_j) \dashrightarrow (\Phi \G/\Phi e_j, (\Phi\tau)_j )$ which preserves the cardinality of the set of markings supported on each vertex. This isomorphism must take the image of the vertex $v_1$ in $\G/e$ to the image of edge $\Phi e_j$ in $\Phi \G/\Phi e_j$, because these are the vertices supporting the images of $\ell$ in $\G/e_j$ and $\Phi \G/ \Phi e_j$, respectively. Thus, if we set $X, Y, Z$ to be the number of markings supported on $v_1, v_2, v_3$, respectively, and put $\Phi X, \Phi Y, \Phi Z$ for the corresponding quantites for $\Phi \G$, we may deduce that $X = \Phi Y + \Phi Z$, so $\Phi Z \leq X$. Similarly, using that there is a weak isomorphism $\varphi_i : (\G/e_i, \tau_i) \dashrightarrow (\Phi \G/\Phi e_i, (\Phi \tau)_i)$ preserving the cardinalities of marking sets and the image of $\ell$, we also deduce that $X + Y = \Phi Z$, so $X \leq \Phi Z$. Thus $X = \Phi Z$ and $Y = \Phi Y = 0$.
	
	Next we observe that in both graphs, edge $j$ does not have any parallel edges, because in this circumstance, $\delta_j(\ell)$ would index a loop in $\Phi \G / \Phi e_j$ which was incident to strictly more loops than the one indexed by $\delta_j(\ell)$ in $\G/e_j$. As $j$ cannot index a $(g, 0)$-bridge in either graph, it must be that $\Phi v_3$ supports a positive number of loops. This implies that $\Phi v_2$ supports no loops, for the same reason why edge $j$ cannot have parallel edges in either graph. So, $\ell$ indexes a loop supported on $\Phi v_3$. Applying the same reasoning and reversing the roles of $\G$ and $\Phi \G$, we see also that edge $i$ cannot have any parallel edges in either graph.
	
	We have thus arrived at a contradiction: in $\Phi\G$, the vertex $\Phi v_2$ must be a vertex of weight $0$ and valence $2$ which supports no markings, which would imply that $\Phi \G$ is not an object of $\Gamma_{g, n}$.
\end{proof}

We are now ready to prove Proposition \ref{FinalReformulation}, and thus the $g \geq 1$ case of Theorem \ref{Main}.
\begin{proof}[Proof of Proposition \ref{FinalReformulation}] We prove each part separately.\\
\noindent\textbf{Part (a)} As in the statement, assume $g\geq 1$, $n \geq 3$, and that we are given $\Phi \in \Aut(\Delta_{g,n})$ such that
		\[\Phi[\B^{g, 0}_{x}, \pi] = [\B^{g, 0}_{x}, \pi] \]
		for all $x \in I_n$, where $\pi : E(B^{g, 0}) \to [g]$ is any edge-labelling. Then, by Lemma \ref{Relationg0}, we also have 
		\[\Phi[\B^{g, 0}_{A}, \pi] = [\B^{g, 0}_A, \pi]\] for all $(g, 0)$-stable subsets $A \subseteq I_n$. Pick an arbitrary $(g, n)$-admissible triple $(k, \ell, B)$ with $(k, \ell) \neq (g, 0)$, together with an edge-labelling $\theta: E(B^{k, \ell}) \to [g]$, to show that $\Phi$ fixes $[\B^{k, \ell}_{B}, \theta]$. Replacing $B$ by $B^c$ if necessary, we may assume $k \geq \ell$. Choose $\Phi(B)$ such that $|\Phi(B)| = |B|$ and such that 
		\[\Phi [\B^{k, \ell}_{B}, \theta] = [\B^{k, \ell}_{\Phi(B)}, \theta]; \]
		then $\Phi(B)$ is uniquely determined unless $(k, \ell) = (0,0)$, by Corollary \ref{SigmaExistence}. Now, using the statement of Lemma \ref{Relation} and letting $A$ vary over all subsets of $I_n$ of size $n - 2$ (note that when $|A| = n - 2$, $(g, 0, A)$ is always $(g, n)$-admissible), we see that 
		\[ \binom{B|B^c}{2} = \binom{\Phi(B)|\Phi(B)^c}{2}, \]
		so we either have $\Phi(B) = B$ or $\Phi(B) = B^c$. Since we must also have $|\Phi(B)| = |B|$, this implies that $\Phi[\B^{k,\ell}_B, \theta] = [\B^{k,\ell}_B, \theta]$ for any $(k,\ell)$-stable $B$, except for the cases where $(k, \ell) \neq (0, 0), (g, 0)$; $n \geq 4$ is even; and $|B| = n / 2$.
		We thus suppose for sake of contradiction that we are given such an admissible triple $(k, \ell, B)$ with $k \geq \ell$ satisfying
		\[\Phi[\B^{k, \ell}_{B}, \theta] = [\B^{k, \ell}_{B^c}, \theta]. \]
		Suppose $x \in B$, and consider the graph $\mathbf{S}$ in Figure \ref{SnailGraph}. Then $\mathbf{S}$ has genus $g$, contains no $3$-cycles, and since $(k, \ell) \neq (0,0), (g, 0)$, and $k \geq \ell$, the graph $\mathbf{S}$ contains no $(g, 0)$-bridges. Thus Lemma \ref{Snails} applies; given any edge-labelling $\tau: E(\mathbf{S}) \to [g + 1]$, there is a weak isomorphism \[\varphi: (\mathbf{S}, \tau) \dashrightarrow (\Phi\mathbf{S}, \Phi\tau).\]
		
		\begin{figure}[h]
			\centering
			\includegraphics[scale=1]{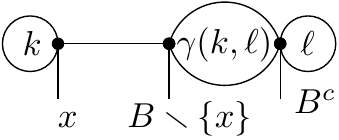}
			\caption{The graph $\mathbf{S}$ in the proof of Proposition \ref{FinalReformulation}. Here $\gamma(k, \ell) = g - k - \ell + 1$.}
			\label{SnailGraph}
		\end{figure}
	
		Let $A_1$ be the set of markings supported on the vertex of $\Phi\mathbf{S}$ which corresponds via the weak isomorphism $\varphi$ to the marking $x$ in $\mathbf{S}$, and similarly let $A_2$ correspond to $B \smallsetminus \{x\}$ and $A_3$ to $B^c$, as in Figure \ref{SnailGraphPrime}. The action of $\varphi$ on $V(\mathbf{S})$ is uniquely determined, so there is no ambiguity in this description.
		\begin{figure}[h]
			\centering
			\includegraphics[scale=1]{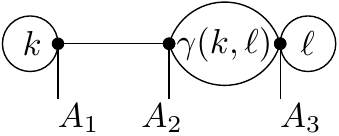}
			\caption{The graph $\Phi\mathbf{S}$ in the proof of Proposition \ref{FinalReformulation}.}
			\label{SnailGraphPrime}
		\end{figure}	
		
		Contracting an edge between $A_2$ and $A_3$ in $\Phi\mathbf{S}$, we see that $A_1 = \{x\}$ and $A_2 \cup A_3 = I_n \smallsetminus \{x\}$, since $\Phi$ fixes $[\B^{k, g - k}_{x}, \pi]$ for any edge-labelling $\pi$. But, if we are to have $\Phi[\B^{k, \ell}_B, \theta] = [\B^{k, \ell}_{B^c}, \theta]$, then we must have $A_1 \cup A_2 = B^c$. Thus $x \in B^c$, which contradicts $x \in B$. This concludes the proof of part (a).\\
		\noindent \textbf{Part (b)} Now suppose that $g \geq 2$, $n = 2$, and that $\Phi \in \Aut(\Delta_{g,2})$ satisfies 
		\[ \Phi[\B^{g-1, 1}_x, \pi] = [\B^{g-1, 1}_x, \pi] \]
		for any edge-labelling $\pi : E(B^{g-1, 1})\to [g]$ and $x = 1, 2$, to show that 
		\[\Phi[\B^{k, \ell}_x, \pi] = [\B^{k, \ell}_x, \pi] \]
		whenever $(k, \ell, \{x\})$ is a $(g, 2)$-admissible triple. Note that the above equality holds automatically when $(k, \ell) = (0, 0)$ by Corollary \ref{SigmaExistence}, so we need only prove it when $(k, \ell) \neq (0, 0), (g-1, 1)$, and $k \geq \ell$. Moreover, it suffices to treat the case $x = 1$. We argue in two cases: first, when $k + \ell = g$, then when $k + \ell < g$.
		
		When $k + \ell = g$, $(k, \ell) \neq (g-1, 1)$ and $k \geq \ell$, we must have $k, \ell \geq 2$, so in particular, the graph $\G$ of Figure \ref{FirstGraphFor2} is stable.
		\begin{figure}[h]
			\centering
			\includegraphics[scale=1]{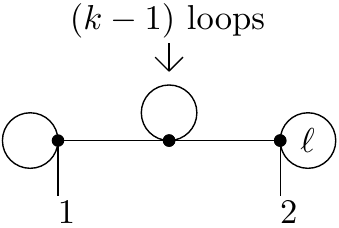}
			\caption{The graph $\mathbf{G}$ in the proof of Proposition \ref{FinalReformulation}.}
			\label{FirstGraphFor2}
		\end{figure}
		Clearly Lemma \ref{Snails} applies to $\G$, so for any edge-labelling $\tau: E(\G) \to [g + 1]$, there is a weak isomorphism $(\G, \tau) \dashrightarrow (\Phi\G, \Phi\tau)$. Since the simplex obtained from $[\G, \tau]$ by contracting the edge supporting the mark $2$ must be fixed, and since $\Phi$ preserves $S_2$-orbits, we may conclude that $[\Phi\G, \Phi\tau] = [\G, \tau]$, so upon contracting the nonloop edge of $\G$ containing the marking $1$, we get the desired result for $\B^{k, \ell}_{1}$.
		
		To treat the case $k + \ell < g$, we modify the graph $\mathbf{S}$ of Figure \ref{SnailGraph}, replacing $x$ by $1$, $B \smallsetminus \{x\}$ by $\varnothing$, and $B^c$ by $2$, and argue similarly, using the result for $k + \ell = g$.
\end{proof}	

\subsection{The case $g = 0$, $n \geq 5$.}\label{GenusZero} It is due to Abreu and Pacini ~\cite{AbreuPacini} that $\Aut(\Delta_{0,n}) \cong S_n$ when $n \geq 5$, while $\Aut(\Delta_{0,4}) \cong S_3$ (we saw this in Example \ref{SmallExample}). For completeness, we conclude by indicating how their result follows from the framework established herein.

When $g = 0$ and $n \geq 5$, the only $(0, n)$-admissible triples are $(0,0, A)$ for $A \subseteq I_n$ satisfying $2 \leq |A| \leq n - 2$. The graph $\B_A\defeq \B_A^{0,0}$ has one edge and two vertices, and therefore has a unique edge-labelling; put $[\B_A]$ for the corresponding $0$-simplex of $\Delta_{0,n}$ Continue to use the relation $\leftrightarrow$ on $\Gamma_{g, n}$ where $\G \leftrightarrow \H$ if and only if there exists $\G'$ with morphisms to both $\G$ and $\H$. The following lemma is the combinatorial manifestation of well-known results on the intersections of boundary divisors of $\overbar{\mathcal{M}}_{0, n}$ ~\cite{Keel}. 
\begin{lem}\label{M0nintersections}
	Suppose $n \geq 5$, and $A, B \subseteq I_n$ are distinct subsets satisfying $2 \leq |A|, |B| \leq n - 2$. Then $\B_A \leftrightarrow \B_B$ if and only if $A \subseteq B$, $B \subseteq A$, $A \subseteq B^c$, or $B^c \subseteq A$.
\end{lem}
When $n \geq 5$, we can use Lemma \ref{M0nintersections} to control the action of $\Aut(\Delta_{0,n})$ on the vertices $[\B_A]$ for $|A| = 2$, as in the following result.
\begin{lem}\label{ExtractingPermutationZero}
	Suppose $n \geq 5$, and suppose $A, B \in \binom{I_n}{2}$. Then there are unique $\Phi(A), \Phi(B) \in \binom{I_n}{2}$ such that $\Phi[\B_A] = [\B_{\Phi(A)}]$ and $\Phi[\B_B]= [\B_{\Phi(B)}]$. Moreover, we have
	\[|A \cap B| = |\Phi(A) \cap \Phi(B)|. \]
\end{lem}
\begin{proof}
That we can find unique such $\Phi(A)$ and $\Phi(B)$ follows from Corollary \ref{SigmaExistence} and the fact that $n \geq 5$. Of course if $A = B$, then $\Phi(A) = \Phi(B)$, so the lemma is true when $|A \cap B| = 2$. When $|A \cap B| < 2$, Lemma \ref{M0nintersections} implies that $\B_A \leftrightarrow \B_B$ if and only if $|A \cap B| = 0$. Since $\Phi$ preserves the relation $\leftrightarrow$, the lemma follows.
\end{proof}
We now prove the remaining cases of Theorem \ref{Image} and Theorem \ref{Main}, which occur when $g = 0$ and $n \geq 5$.
\begin{prop}\label{FinalZero}
Suppose $\Phi \in \Aut(\Delta_{0,n})$ with $n \geq 5$. Then there exists a unique $\sigma \in S_n$ such that 
\[\Phi|_{\V^{2}_{0, n}} = f_{\sigma}|_{\V^{2}_{0, n}}, \]
so $\Aut(\Delta_{0,n}) \cong S_n$ for $n \geq 5$.
\end{prop}
\begin{proof}
	First we claim that when $n \geq 5$, any permutation
	\[ \Psi: \binom{I_n}{2} \to \binom{I_n}{2} \]
	satisfying $|A \cap B| = |\Psi(A) \cap \Psi(B)|$ for all $A, B \in \binom{I_n}{2}$ is uniquely induced by some $\sigma \in S_n$. The permutation $\sigma$ is extracted from $\Psi$ by setting
	\[\sigma(x) = \Psi(J_1) \cap \Psi(J_2) \]
	for any two distinct $J_1, J_2 \in \binom{I_n}{2}$ which both contain $x$. If we put $K_n$ for the complete graph on $n$ vertices, the claim is equivalent to the fact that for $n \geq 5$, any $\Psi \in \Aut(K_n) \cong S_n$ is uniquely determined by a permutation of the edges which preserves the number of vertices shared by two edges; the details are left to the reader.
	
	Therefore, Lemma \ref{ExtractingPermutationZero} implies that there exists a unique permutation $\sigma$ such that $\Phi[\B_A] = f_\sigma[\B_A]$ for all $A \subseteq I_n$ with $|A|= 2$. By complementation, it follows that $\Phi[\B_A] = f_\sigma[\B_A]$ when $|A| = n - 2$ as well. Composing $\Phi$ with $f_{\sigma^{-1}}$, it thus suffices to prove that if $\Phi$ fixes $[\B_A]$ for all $|A| = 2$, then $\Phi$ fixes all $[\B_B]$ with $3 \leq |B| \leq n - 3$. Fix such $B$, and find $\Phi(B) \subseteq I_n$ such that
	\[\Phi[\B_B] = [\B_{\Phi(B)}].\] 
	Using Lemma \ref{M0nintersections}, we see that
	\[\binom{B|B^c}{2} = \left\{A \in \binom{I_n}{2} \mid \B_A \leftrightarrow \B_B \right\} .\]
	Since $\Phi$ preserves $\leftrightarrow$ and fixes the simplices $[\B_A]$ for $|A| = 2$, it follows that 
	\[\binom{B|B^c}{2} = \binom{\Phi(B)|\Phi(B)^c}{2}, \]
	so $\Phi(B) = B$ or $\Phi(B) = B^c$. Either way, we may conclude that $\Phi[\B_B] = [\B_B]$, and the proof is finished. 
\end{proof}

\bibliographystyle{alpha}
\bibliography{deltabib}
\end{document}